\documentclass[10pt,letterpaper]{article}

\newif\ifarxiv
\arxivtrue

\ifarxiv
\usepackage[top=1in,left=1in,footskip=0.75in]{geometry}
\else
\usepackage[top=0.85in,left=2.75in,footskip=0.75in]{geometry}
\fi

\usepackage{amsmath,amssymb}

\usepackage{changepage}

\usepackage[utf8x]{inputenc}

\usepackage{textcomp,marvosym}

\usepackage{cite}

\usepackage{nameref,hyperref}

\usepackage[right]{lineno}

\usepackage{microtype}
\DisableLigatures[f]{encoding = *, family = * }

\usepackage[table]{xcolor}

\usepackage{array}

\newcolumntype{+}{!{\vrule width 2pt}}

\newlength\savedwidth

\ifarxiv
\else
\raggedright
\fi

\usepackage[aboveskip=1pt,labelfont=bf,labelsep=period,justification=raggedright,singlelinecheck=off]{caption}

\makeatletter
\renewcommand{\@biblabel}[1]{\quad#1.}
\makeatother

\usepackage{lastpage,fancyhdr,graphicx}
\usepackage{epstopdf}
\fancyheadoffset[L]{2.25in}
\fancyfootoffset[L]{2.25in}
\usepackage{hyperref}

\usepackage{cite}

 \usepackage{algorithmic}

\usepackage{amsmath}
\usepackage{amssymb}
\usepackage{amstext}

\usepackage{color}
\usepackage{url}

\renewcommand{\vec}[1]{ {\mathbf{#1}}}

\usepackage{multirow}
\newcommand{\T}{^\top}

\newcommand{\bgamma}{\boldsymbol{\gamma}}

\newcommand{\bl}{\mbox{\boldmath $l$}}

\newcommand{\bd}{\mbox{\boldmath $d$}}

\newcommand{\bu}{\mbox{\boldmath $u$}}

\usepackage{mdframed}

{\end{tabbing} \end{mdframed} \vspace{5px}}

\newcommand{\BM}{\begin{bmatrix}}
\newcommand{\EM}{\end{bmatrix}}

\newcommand{\dd}[2]{\frac{{\rm d}#1}{{\rm d}#2}}

\newcommand{\ddt}{\dd{}{t}}

\newcommand{\vH}{\vec{H}}

\newcommand{\beq}{\begin{equation}}
\newcommand{\eeq}{\end{equation} }

\newcommand{\A}{\mathbf{A}}
\newcommand{\x}{\mathbf{x}}
\newcommand{\f}{\mathbf{f}}
\newcommand{\M}{\mathbf{M}}
\newcommand{\h}{\mathbf{h}}
\newcommand{\delt}[2][]{\ifthenelse{\equal{#1}{}} { \delta \mathbf{#2}}{\delta \mathbf{#2}_{#1}}}

\newcommand{\llambda}{{\boldsymbol{\lambda}}}

\newcommand{\Rn}{\mathbb{R}^n}
\newcommand{\R}{\mathbb{R}}
\renewcommand{\H}{\mathbf{H}}
\newcommand{\z}{\mathbf{z}}
\newcommand{\y}{\mathbf{y}}
\newcommand{\g}{\mathbf{g}}
\newcommand{\X}{\mathbf{X}}

\usepackage{graphicx}
\usepackage{empheq}
\usepackage{amsthm}
\usepackage{bbding}

\usepackage{mathrsfs}

\newtheorem{theorem}{Theorem}
\newtheorem{definition}{Definition}
\newtheorem{example}{Example}
\newtheorem{proposition}{Proposition}
\newtheorem{remark}{Remark}
\newtheorem{corollary}{Corollary}

\newcommand{\script}[1]{\mathscr{#1}}

\newcommand{\logdet}[1]{{\rm logdet}(#1)}
\newcommand{\tr}[1]{{\rm tr}(#1)}

\newcommand{\Z}{\mathbf{Z}}
\newcommand{\grad}[2]{ \nabla_{#2} #1 } 
\newcommand{\hess}[3]{\nabla_{#2#3} #1 }

\newcommand{\pd}[2]{ \partial_{#2} #1 } 
\newcommand{\pdd}[3]{\partial_{#2#3} #1 }

\newcommand{\Jac}[2]{ \frac{\partial #1}{\partial #2} }

\newcommand{\jac}[2]{\Jac{#1}{#2}}

\newcommand{\xsubi}{i}
\newcommand{\xsubj}{j}
\newcommand{\xsubk}{k}
\newcommand{\xsub}[1]{#1}

\newcommand{\nextTime}[1]{}

\renewcommand{\a}{\mathbf{a}}
\newcommand{\ah}{\hat{\a}}
\newcommand{\Y}{\mathbf{Y}}
\newcommand{\J}{\mathbf{J}}

\newcommand{\Jh}{\hat{\J}}
\renewcommand{\u}{\mathbf{u}}
\newcommand{\W}{\mathbf{W}}
\renewcommand{\A}{\mathbf{A}}
\newcommand{\Acal}{\mathcal{A}}
\newcommand{\at}{\tilde{\a}}
\newcommand{\s}{\mathbf{s}}
\newcommand{\xt}{\tilde{\x}}
\newcommand{\Jdh}{\hat{\dot{\J}}}
\newcommand{\K}{\mathbf{K}}
\renewcommand{\P}{\mathbf{P}}
\newcommand{\Q}{\mathbf{Q}}

\begin{document}
\nolinenumbers
\vspace*{0.2in}

\ifarxiv
	\graphicspath{{./}}
\else
	\graphicspath{{Figures/}}
\fi

\begin{flushleft}
{\Large
\textbf\newline{Beyond Convexity - Contraction and Global Convergence of Gradient Descent} 
}
\newline
\\
Patrick M. Wensing\textsuperscript{1*},
Jean-Jacques Slotine\textsuperscript{2}
\\
\bigskip
\textbf{1} Department of Aerospace and Mechanical Engineering, University of Notre Dame, Notre Dame, IN, USA
\\
\textbf{2} Department of Mechanical Engineering, Department of Brain and Cognitive Sciences, and Nonlinear Systems Laboratory, Massachusetts Institute of Technology, Cambridge, MA, USA
\\
\bigskip

%
%





* pwensing@nd.edu

\end{flushleft}

\section*{Abstract}
This paper considers the analysis of continuous time gradient-based optimization algorithms through the lens of nonlinear contraction theory.   
It demonstrates that in the case of a time-invariant objective, most elementary results on gradient descent based on convexity can be replaced by much more general results based on contraction.
In particular, gradient descent converges to a unique equilibrium if its dynamics are contracting in any metric, with convexity of the cost corresponding to the special case of contraction in the identity metric. 
More broadly, contraction analysis provides new insights for the case of geodesically-convex optimization, wherein non-convex problems in Euclidean space can be transformed to convex ones posed over a Riemannian manifold. In this case, 
natural gradient descent converges to a unique equilibrium if it is contracting in any metric, with geodesic convexity of the cost corresponding to contraction in the natural metric. New results using semi-contraction provide additional insights into the topology of the set of optimizers in the case when multiple optima exist.
Furthermore, they show how semi-contraction may be combined with specific additional information to reach broad conclusions about a dynamical system. 
The contraction perspective also easily extends to time-varying optimization settings and allows one to recursively build large optimization structures out of simpler elements. Extensions to natural primal-dual optimization and game-theoretic contexts further illustrate the potential reach of these new perspectives.


\section{Introduction}

This paper considers the analysis of continuous-time gradient-based optimization through the lens of nonlinear contraction theory. It is  motivated, in part, by recent observations in machine learning that arise in the application of gradient descent (or its stochastic counterpart) for the training of over-parameterized networks \cite{bassily2018exponential}. Modern networks often possess many more parameters than training examples and can fit the labels perfectly, resulting in submanifold valleys of the parameter space with equal cost \cite{cooper2018loss,liu2020theory,brea2019weight}. Moreover, recent results  suggest that highly-redundant networks experience few to no local optima that are not global optima \cite{sagun2017empirical, allen2018convergence,du2018gradient}. These observations may be surprising in light of the fact that the loss landscapes for these problems are rarely convex.

Although convex problems admit provable globally optimal solutions, other broader classes of functions share this same property. For example, Invex functions \cite{hanson1981sufficiency} guarantee that any local optimum is a global optimum, although the utility of invexity conditions remains a point of contention \cite{zualinescu2014critical}. Functions satisfying the Polyak-Lojasiewicz (PL) inequality \cite{bassily2018exponential,polyak1963gradient,karimi2016linear,liu2020theory} give rise to exponentially convergent gradient descent to  a provably optimal solution. While the PL condition is, in general, difficult to verify without an a-priori known globally optimal solution, the existence of zero-loss solutions in over-parameterized learning \cite{liu2020theory,du2018gradient,allen2018convergence} makes it tractable in important special cases. Geodesic convexity \cite{RAPCSAK91,Sra_First_Order_Methods} generalizes convexity to a Riemannian setting, with applicability to optimization on manifolds \cite{Absil08}, as well as to conventional Euclidean settings where $\mathbb{R}^n$ is endowed with a manifold structure through the definition of a metric. Here, we consider another class of conditions for the convergence of gradient and natural gradient descent to a globally optimal point. We do so through adopting the perspective of nonlinear contraction theory and analyzing gradient descent in continuous time. 

Contraction theory~\cite{Slotine98} allows the stability of nonlinear non-autonomous systems to be characterized through linear time-varying dynamics describing the propagation of infinitesimally small displacements along the systems' flow. The existence of a Riemannian metric that contracts these virtual displacements (i.e., elements in the tangent space) is necessary and sufficient for exponential convergence of any pair of trajectories. Contraction naturally yields methods for constructing stable systems of systems, including synchronization phenomena \cite{tabareau} and consensus \cite{Wei_delay,cloud} as well as other key building blocks that allow the construction of large contracting systems out of simpler elements \cite{modular}. These properties provide opportunities to construct larger optimization structures from simpler elements (e.g., in distributed or competitive optimization settings).

    The contribution of this paper is to apply these contraction tools for the analysis of gradient and natural gradient optimization. We consider optimization problems posed over $\Rn$ wherein no explicit manifold structure necessarily exists a-priori. Instead, we consider the analysis of optimization following endowing these problems with additional structure (a Riemannian metric), analyzing their convergence, and considering the use of contraction tools to build larger optimization structures out of smaller ones. Analysis proceeds in continuous time. While this approach is limited, in part, by the fact that computational optimization algorithms require a discrete implementation, a  continuous perspective has yielded insight on important phenomena such as in the analysis~\cite{su2014differential}, discrete implementations \cite{Zhang18}, and extensions \cite{wibisono2016variational,krichene2015accelerated} of Nesterov's accelerated gradient descent method \cite{nesterov1998introductory}. It has also enabled analysis of primal-dual algorithms~\cite{Kostya}, where an absolute time reference is obtained by introducing additional fast dynamics or delays using a singular perturbation framework. Recent results \cite{frana2020dissipative} provide principled tools to derive discrete-time implementations that preserve specific continuous-time convergence rates. 

The paper is organized as follows.
Section~\ref{sec:Convex} provides our main results, detailing the applicability of contraction theory to analyze gradient descent in continuous time. We show that convex functions represent the special case of contraction in the identity metric. The flexibility afforded by state-dependent contraction metrics, however, enables significant extra freedom for guaranteeing that all local optima are globally optimal. We then consider the extensions of these results to natural gradient descent, where geodesic convexity of a function corresponds to contraction of its natural gradient system in the natural metric. In both cases, results highlight the topology of the set of optimizers in the case of semi-contraction, which would have most direct applicability to over-parameterized networks. New results also show how semi-contraction may be combined with specific additional information to reach broad conclusions about a dynamical system. Section~\ref{GPD_opt} details extensions of these results to the case of primal-dual type dynamics that appear in mixed convex/concave saddle systems, and shows how a broad class of natural adaptive control laws can be interpreted as a primal-dual system. Section \ref{sec:application} discusses the special case of g-convex functions and associated combination properties for interfacing with other models. Section \ref{sec:Conclusions} provides an outlook on potential future advances that may stem from these connections. 

\section{Contraction Analysis of Gradient Systems}
\label{sec:Convex}

We first recall basic definitions and facts on convex optimization and show how a contraction analysis of gradient-based optimization considerably generalizes the class of functions that admit a unique global optimum. Following this presentation, results are generalized to the case of geodesically-convex optimization, which is particularly suited to analysis via contraction tools. Throughout this analysis, given a differentiable function $\h :\R^n \rightarrow \R^m$, we denote the Jacobian of $\h(\x)$ by
\[
\Jac{\h}{\x} = \begin{bmatrix} \Jac{\h}{x_1} & \cdots & \Jac{\h}{x_n} \end{bmatrix}  \in \R^{m\times n}
\]
In the special case of a scalar-valued function $f:\R^n\rightarrow \R $ we denote the gradient of $f(\x)$ by
\[
\grad{f}{}(\x) = \left[ \Jac{f}{\x} \right]\T \in \R^n
\]
and its Hessian by $\nabla^2f(\x)$. Unless otherwise stated, we assume all functions are sufficiently smooth such that derivatives of the necessary order exist and are continuous. 

Before we embark on this discussion, let us note that of course,
as illustrated, e.g., in~\cite{su2014differential} and in the following example, continuous-time analysis tools in general may be used to conceptually illuminate the mechanisms involved in discrete-time algorithms. As this paper will show, contraction tools give particularly simple insights into important classes of optimization problems, such as, e.g., geodesically-convex optimization.

\begin{example}
The Polyak-Lojasiewicz (PL) inequality is one of the most general sufficient conditions for discrete-time gradient descent
to exhibit linear convergence rates without strong convexity of the
cost~\cite{polyak1963gradient,karimi2016linear}.
A function is said to satisfy the PL inequality if it has a (typically unknown) global minimum value $f^*$ 
and there exists a constant $\mu>0$ such that
\[
\forall \x, \ \ \ \ \ \ \ \| \nabla f(\x) \|^2 \ \ge \ \mu ( f(\x) - f^*) 
\]

Consider gradient descent on the cost function $f(\x)$ from a continuous-time point of view, 
\[
\dot{\x} = - \nabla{f}({\x})
\]
Using $ \ V= f(\x) - f^* \ $ as a Lyapunov-like function, and then requiring that $V$ converges exponentially with rate $\mu$,
yields
\[
\dot{V} = - \| \nabla f(\x) \|^2 \ \le \ - \mu V
\]
The inequality above is exactly the PL condition. Thus, we see that the PL condition is nothing but the condition
for exponential convergence of the residual cost $ \ V= f(\x) - f^* \ $.

Similarly, imposing $\ \dot{V} \le - \mu \sqrt{V}\ $, corresponding to finite-time convergence (in time less than $\ 2 \sqrt{V(0)}/\mu \ $ \cite{slotine1991applied}), would require a modified PL-like condition
\[
\forall \x, \ \ \ \ \ \ \ \| \nabla f(\x) \|^2 \ \ge \ \mu \ \sqrt{ f(\x) - f^*}
\]
while imposing $\ \dot{V} \le - \mu V^2\ $ would require
\[
\forall \x, \ \ \ \ \ \ \ \| \nabla f(\x) \|\ \ge \ \sqrt{\mu} \ (f(\x) - f^*)
\]
\end{example}

By comparison, the results pursued via contraction analysis in this paper will ensure exponential convergence of any pair of trajectories for gradient descent, but likewise will ensure convergence of those solutions to a global optimum.

\subsection{Relationships Between Convexity and Contraction}
\label{sec:contractAndGradient}

\begin{definition}[Strong Convexity] 
\label{def:strongConv}
A twice differentiable function $f : \mathbb{R}^n \rightarrow \mathbb{R}$ is $\alpha$-strongly convex with $\alpha>0$ if its Hessian matrix $\nabla^2{f}{}{}(\x)$ satisfies the matrix inequality
\[
\nabla^2{f}(\x) \succeq \alpha\, \mathbf{I} \quad \quad \forall \x \in \mathbb{R}^n
\]
\end{definition}
As its name suggests, a function that is strongly convex is convex in the usual sense, while the converse is not always true. From a dynamic systems perspective, strong convexity provides exponential convergence of gradient flows:
\begin{proposition}[Exponential Convergence of Gradient Systems for Strongly Convex Functions]
\label{prop:euc_grad_desc}
\ If a twice differentiable function $f: \mathbb{R}^n \rightarrow \mathbb{R}$ is $\alpha$-strongly convex, then its gradient system
\begin{equation}
\dot{\x} = - \grad{f}{}(\x)
\label{eq:grad_descent}
\end{equation}
converges to the unique global minimum of $f$ exponentially with rate $\alpha$.
\end{proposition}

Toward proving this proposition, we will consider stability analysis through the application of nonlinear contraction theory.

\begin{definition}[Contraction Metric~\cite{Slotine98}]
\label{def:contraction}
A system $\dot{\x} = \h(\x,t)$ is said to be contracting at rate $\alpha>0$ with respect to
a symmetric positive definite metric $\M: \Rn  \rightarrow \R^{n\times n}$,  if for all  $t \in \R$ and all $\x \in \Rn $,
\begin{equation}
\dot{\M} + \A\T\, \M + \M\, \A \preceq -2 \alpha \M
\label{eq:contraction_condition}
\end{equation}
where $\A(\x,t) = \Jac{\h}{\x}$ is the system Jacobian and $\dot{\M}= \sum_i \left( \partial \M / \partial x_i \right) h_i(\x,t)$.
The system is said to be semi-contracting with respect to $\M$ when \eqref{eq:contraction_condition} holds with $\alpha=0$. 
\end{definition}

Given an $\alpha$-contracting system and an arbitrary pair of initial conditions $\x_1(0)$ and $\x_2(0)$, the solutions $\x_1(t)$ and $\x_2(t)$ converge to one another exponentially
\begin{equation}\label{d_M}
d_{\mathcal{M}}( \x_1(t), \x_2(t)) \ \le \ {\rm e}^{-\alpha t} \ d_{\mathcal{M}}(\x_1(0), \x_2(0)) \ \ 
\end{equation}
where $d_{\mathcal{M}}(\cdot,\cdot)$ denotes the geodesic distance on the Riemannian manifold $\mathcal{M} = (\Rn, \M)$. This property can be shown by considering the evolution of differential displacements $\delta \x$, which describe the evolution of nearby trajectories and coincide with the notion of virtual displacements in Lagrangian mechanics. More precisely, letting $\x(t;\x_0, t_0)$ denote the solution of $\dot{\x} = \h(\x,t)$ from initial condition $\x(t_0) = \x_0$, differential displacements evolve according to
\[
\delta \x(t) = \frac{ \partial \x(t; \x_0, t_0) }{\partial \x_0} \ \delta \x(t_0)
\]
Property~\eqref{d_M} follows from the evolution of the squared length of these differential displacements~\cite{Slotine98}, which verifies
\begin{equation}\label{ddt_z^2}
    \ddt (\delta \x\T \M \delta \x) \le - 2 \alpha (\delta \x\T \M \delta \x) 
\end{equation}
Furthermore, if a system is $\alpha$-contracting in a metric $\M$ that satisfies $\M(\x)\succeq \beta \mathbf{I}$  uniformly
for some constant $\beta>0$, then any two solutions verify
\[
\| \x_1(t) - \x_2(t) \| \ \le \ \frac{1}{\sqrt{\beta}} \  {\rm e}^{-\alpha t} d_{\mathcal{M}}(\x_1(0), \x_2(0))
\]

\begin{example}
Consider an $\alpha$-strongly convex function $f$ and its associated gradient descent system~\eqref{eq:grad_descent}. Since $f$ is strongly convex, it has a unique global minimum $\x^*$, which is a equilibrium point of \eqref{eq:grad_descent}. It can be verified that the gradient descent dynamics of $f$ are contracting in the identity metric $\M = {\bf I}$ with rate $\alpha$. Since geodesic distances are just Euclidean distances in this metric,~\eqref{d_M} immediately implies that
\[
\forall t \ge 0, \ \ \|\x(t) - \x^*\| \ \le \ {\rm e}^{-\alpha t} \ \| \x(0) - \x^* \| 
\]
thus proving Proposition~\ref{prop:euc_grad_desc}.
\end{example}

From this example, it is clear that strongly convex functions are a special case of ones whose gradient systems are contracting. The following proposition shows that one does not lose the convergence properties to a global optimum on this more general class of functions.

\begin{proposition}[Exponential Convergence of Contracting Gradient Systems]
\label{prop:contracting_grad_desc} Consider again gradient descent as in equation \eqref{eq:grad_descent}. The system converges exponentially to a unique global minimum if it is contracting in \emph{some} metric.
\end{proposition}

\begin{proof} Because \eqref{eq:grad_descent} is autonomous and contracting, it converges exponentially to a unique equilibrium $\x^\star$ \cite{Slotine98}. Furthermore, this equilibrium must be a global minimum since $f$ can only decrease along trajectories, with $\ \dot{f} = -  \grad{f}{}(\x)\T \ \grad{f}{}(\x) \ < \ 0\ $ for 
$\ \x \neq \x^\star .$
\end{proof}

The above result, which emphasizes contraction rather than convexity
as a sufficient condition to converge to a global minimum, can be extended to the semi-contracting case as follows.

\begin{proposition}[Asymptotic Convergence of Semi-Contracting Gradient Systems]
\label{prop:SemiGrad}
Consider a twice differentiable function $f : \mathbb{R}^n \rightarrow \mathbb{R}$, a symmetric positive definite metric $\M: \mathbb{R}^n \rightarrow \mathbb{R}^{n \times n}$, and the associated gradient system
\begin{equation}
\dot{\x} = -\grad{f}{}(\x) \label{eq:grad_auton}    
\end{equation}
Assume that dynamics \eqref{eq:grad_auton} is semi-contracting in \emph{some}
metric, and furthermore that one trajectory of the system is known to be bounded. Then, (a) $f$ has at least one stationary point, (b) any local minimum of $f$ is a global minimum, (c) all global minima of $f$ are path-connected, and (d) all trajectories asymptotically converge
to a global minimum of $f$.
\end{proposition}

\begin{proof}
(a) By assumption, there exists some initial condition $\x_0$ such that $\x(t; \x_0)$ remains bounded. This, in turn, implies that the $\omega$-limit set $\omega[\x(t;\x_0)]$ is non-empty, compact, forward invariant, and that
\[
d(\x(t;\x_0), \omega[\x(t;\x_0)]) \rightarrow 0 ~{\rm \ as~ \ }t\rightarrow + \infty
\]
Let $\x^*$ denote an element of $\omega[\x(t;\x_0)]$. Since \eqref{eq:grad_auton} is a gradient system, Theorem 15.0.3 of \cite{Wiggins2003} guarantees that $\x^*$ must be an equilibrium point of \eqref{eq:grad_auton}. This proves that $f$ has at least one stationary point.

Let us now show that $\omega[\x(t;\x_0)]$ consists only of the single point $\x^*$, by contradiction. Let $\x_1^*$ and $\x_2^*$ be distinct elements in $\omega[\x(t;\x_0)]$. Further let $\epsilon = d_\M(\x_1^*, \x_2^*)$ the geodesic distance between $\x_1^*$ and $\x_2^*$. Then, the geodesic balls $\mathcal{B}_1 : = \mathcal{B}_{\M}(\x_1^*,\frac{\epsilon}{3})$ and $\mathcal{B}_{\M}(\x_2^*,\frac{\epsilon}{3})$ are disjoint. Further, since the system is semi-contracting these geodesic balls are forward invariant. Yet, since $\x_1^*$ a limit point, $\x(t;\x_0)$ arrives within $\mathcal{B}_1$ at some point, and never leaves. Likewise, since $\x_2^*$ is a limit point, $\x(t;\x_0)$ arrives within $\mathcal{B}_2$ at some point, and never leaves. Thus, we have a contradiction, and the limit set must consist of a single point.

(b) and (c): Consider now two equilibrium points of \eqref{eq:grad_auton}, $\x_1^*$ and $\x_2^*$ , and a smooth
path $\bgamma(s)$ such that $\bgamma(0) = \x^*_1$ and $\bgamma(1) = \x^*_2$. Since the gradient dynamics are semi-contracting, for each $s$ the solution $\x(t;\bgamma(s))$ remains bounded. Thus, by the same reasoning as above, each $\x(t;\bgamma(s))$ converges to some equilibrium $\x^*(s)$ as $t\rightarrow + \infty$. Since $\nabla f(\x^*(s)) =0$ for each $s$, and $\x^*(s)$ smoothly connects $\x^*_1$ and $\x_2^*$, it follows that $f(\x_1^*) = f(\x_2^*)$. That is, all solutions converge to the same value for $f$.

(d): That all solutions of \eqref{eq:grad_descent} asymptotically converge to a global minimum of $f$ follows from that fact that $f$ decreases along all solutions, and all solutions converge to the same value for $f$.
\end{proof}

\begin{remark} In the case that a contraction metric needs to be found numerically, note that the conditions \eqref{eq:contraction_condition} for certifying contraction or semi-contraction are convex criteria. Thus, in many instances, the process of finding a metric numerically to verify contraction may be accomplished via convex optimization approaches, such as those based on sums-of-squares programming~\cite{aylward2008stability}.  
\end{remark}

\subsection{Relationship Between Geodesic Convexity and Contraction}

Geodesic convexity \cite{RAPCSAK91} generalizes conventional notions of convexity to the case where the domain of a function is equipped with a Riemannian metric. A special case occurs in geometric programming (GP)~\cite{Boyd_GeomProgram}. In GP, a non-convex problem over positive variables $\{x_i\}_{i=1}^N$ can be transformed into a convex problem by a change of variables $y_i = {\rm log}(x_i)$. Alternately GP can be formulated over the positive reals viewed as a Riemannian manifold by measuring differential length elements $ds$ in a relative sense
\begin{equation}
ds^2 = \sum_{i=1}^N \left( \frac{dx_i}{x_i} \right)^2 = \sum_{i=1}^N dy_i^2
\label{eq:GP_Gconvex}
\end{equation}
Geodesically-convex optimization generalizes this transformation strategy to a broader class of problems~\cite{Sra_First_Order_Methods}. However, beyond special cases (see, e.g., \cite{SRA_Conic_Optimization}), generative procedures remain lacking to formulate g-convex optimization problems or recognize g-convexity. 

To introduce g-convexity more formally, consider a function $f:\mathbb{R}^n \rightarrow \mathbb{R}$ and a positive definite metric $\M:\Rn \rightarrow \R^{n\times n}$. We note that geodesic convexity of $f$ is not an intrinsic property of the function itself, but rather is a property of $f$ defined on the Riemannian manifold $(\Rn, \M)$.  
\begin{definition}[g-Strong Convexity \cite{udriste}]
A twice differentiable function $f :  \mathbb{R}^n \rightarrow \mathbb{R}$ is said to be geodesically $\alpha$-strongly convex (with $\alpha>0$) in a symmetric positive definite metric $\M$ if its Riemannian Hessian matrix $\H(\x)$ satisfies:
\begin{equation}
\H(\x) \succeq \alpha\, \M(\x)\quad \quad \forall \x \in \mathbb{R}^n
\label{eq:Hessian}
\end{equation}
The elements of the Riemannian Hessian are given as \cite{udriste}
\begin{equation}
H_{ij} = \pdd{f}{i}{j} - \Gamma_{ij}^k \, \pd{f}{k} 
\label{eq:riemhessian}
\end{equation}
where $\pdd{f}{i}{j} = \frac{\partial^2 f}{\partial x_i \partial x_j}$ provide the elements of the conventional (Euclidean) Hessian and $\Gamma_{ij}^k$ denotes the Christoffel symbols of the second kind
\[
\Gamma_{ij}^m = \frac{1}{2} \sum_{k=1}^n \left[ M^{mk} \left( \pd{M_{ik}}{j} + \pd{M_{jk}}{i} - \pd{M_{ij}}{k} \right)  \right]
\]
with $M^{ij}(\x) = (\M(\x)^{-1})_{ij}$. The function $f$ is g-convex when \eqref{eq:Hessian} holds with $\alpha=0$.
\end{definition}

The Riemannian Hessian generalizes the notion of the Hessian from a Euclidean context and captures the curvature of $f$ along geodesics. Likewise, the natural gradient generalizes the notion of a Euclidean gradient to the Riemannian context in the following sense.

\begin{definition}[Natural Gradient \cite{Amari98b}]
Consider $\Rn$ equipped with a Riemannian metric $\M$. The {\em natural} gradient of a differentiable function $f : \Rn \rightarrow \R$ is the direction of steepest ascent on the manifold and is given in coordinates by $ \ \M(\x)^{-1} \grad{f}{}(\x) \ $.

\end{definition}

\begin{remark}
When $\M(\x)$ is the Hessian of some twice differentiable strictly convex scalar function $\psi(\x)$, natural gradient descent coincides with the 
continuous-time limit of mirror descent~\cite[Sec.~2.3]{gunasekar2018characterizing} with potential  $\psi(\x)$.
\end{remark}

\begin{remark} From a differential geometric viewpoint, the first covariant derivative of $f$ is a covector field  given in coordinates by $\grad{f}{}(\x)$, while
the natural gradient is a vector field given in coordinates by $\M(\x)^{-1} \grad{f}{}(\x)$ ~\cite{Amari98b}. In a Euclidean context, where $\M(\x)$ is identity, this distinction between covariant (covector) and contravariant (vector) representations of the gradient is immaterial.

Similarly, the Riemannian Hessian $\vH$ represents in coordinates the second covariant derivative of $f$.

\end{remark}

When $\M$ is the identity metric, geodesic $\alpha$-strong convexity naturally coincides with the definition of $\alpha$-strong convexity in Definition~\ref{def:strongConv}. 
 The natural gradient can be used to directly mirror Proposition~\ref{prop:euc_grad_desc} within the Riemannian context.

\begin{theorem}[Equivalence between g-Strong Convexity and Contraction of Natural Gradient]
\label{th_main} Consider a twice differentiable function $f  : \mathbb{R}^n\times \mathbb{R} \rightarrow \mathbb{R}$, a symmetric positive definite metric $\M: \mathbb{R}^n \rightarrow \mathbb{R}^{n \times n}$, and the natural gradient system \cite{Amari98b}
\begin{equation}
\dot{\x} = \h(\x,t) =  -\M(\x)^{-1}\, \grad{f(\x,t)}{\x} \label{eq:NatGrad}
\end{equation}
Then, $f$ is $\alpha$-strongly g-convex in the metric $\M$ for each $t$  if and only if \eqref{eq:NatGrad} is contracting with rate $\alpha$ in the metric $\M$. More specifically, the Riemannian Hessian verifies
\begin{equation}
\vH = -\frac{1}{2}\left( \M \A  + \A\T \M + \dot{\M} \right)
\label{eq:H}
\end{equation}
where $\A = \Jac{\h}{\x}$.

\end{theorem}

Appendix 1 provides a self-contained proof using conventional tensor analysis methods~\cite{lovelock1989tensors}, whose relationship with contraction conditions have been noted previously \cite{lohmiller2013exact,Loh09}. The same relationships drive coordinate-free versions of the result in~\cite{bullo}.

\begin{remark} Theorem~\ref{th_main} can also be viewed as a special case of contraction analysis for complex Hamilton-Jacobi dynamics~\cite{Loh09}. A reorganization of \eqref{eq:NatGrad} as 
\[
\M(\x) \ \dot{\x} \  = - \grad{f(\x,t)}{\x}
\]
may be recognized as the generalized momentum being the negative covariant gradient within a Hamiltonian mechanics context. 
\end{remark}

\begin{remark}
While Theorem \ref{th_main} applies to $\alpha$-strong convexity, the link between the Riemannian Hessian and the contraction condition~\eqref{eq:contraction_condition} also provides immediate equivalence between g-convexity of a function and semi-contraction of its natural gradient dynamics.
\end{remark}

\begin{remark}
Equation \eqref{eq:H} provides an alternate
way to compute the geodesic Hessian $\H$, and, as expected, leaves it invariant when the metric $\M$ is scaled by a strictly positive constant. Because of
the structure of the natural gradient dynamics, scaling
$\M$ is akin to scaling time and implies inversely scaling the contraction rate $\alpha$, consistently with~\eqref{eq:Hessian}.\\
By contrast, note that given a \emph{fixed} dynamics $\h$, the contraction metric analyzing it can always
be arbitrarily scaled while leaving the contraction rate unchanged.
\end{remark}

Similar to in Section~\ref{sec:contractAndGradient} where convexity corresponded to contraction of gradient in the identity metric, we likewise see that Thm.~\ref{th_main} imposes g-convexity via a particular choice of contraction metric for the natural gradient dynamics. Mirroring Prop.~\ref{prop:contracting_grad_desc}, removing this restriction on the contraction metric leads to significant additional flexibility for guaranteeing convergence to a globally optimal point. 

\begin{proposition}[Exponential Convergence of Contracting Natural Gradient Systems]
\label{prop:contracting_Natgrad_desc} Consider again natural gradient descent as in equation \eqref{eq:NatGrad}. The system converges exponentially to a unique global minimum if it is contracting in \emph{some} metric.
\end{proposition}

\begin{proof}
The proof follows immediately from the same logic as the proof of Proposition~\ref{prop:contracting_grad_desc}.
\end{proof}

\begin{remark}
\label{rem:robust}
\renewcommand{\bd}{{\mathbf d}}
Note that contraction also provides robustness. Consider perturbed dynamics $\dot{\x} = \h(\x,t) + \bd(t)$ with $\sqrt{ \bd(t)\T \M(\x) \bd(t) } < R$ uniformly. If the dynamics are contracting with rate $\lambda$, then all trajectories contract to a geodesic ball of radius $R/\lambda$ \cite{Slotine98}. This observation implies favorable properties for algorithms where an exact gradient may be difficult or intractable to compute, with approximation methods used in their place.
\end{remark}

\begin{theorem}[Semi-Contraction for Natural Gradient]
\label{thm:natGradNonStrict}
Consider a twice differentiable function $f : \mathbb{R}^n \rightarrow \mathbb{R}$, a symmetric positive definite metric $\M: \mathbb{R}^n \rightarrow \mathbb{R}^{n \times n}$, and the associated natural gradient system
\begin{equation}
\dot{\x} = - \M(\x)^{-1} \grad{f}{}(\x) \label{eq:nat_grad_auton}    
\end{equation}
Assume that dynamics \eqref{eq:nat_grad_auton} is semi-contracting in \emph{some}
metric, and furthermore that one trajectory of the system is known to be bounded. Then, (a) $f$ has at least one stationary point, (b) any local minimum of $f$ is a global minimum, (c) all global minima of $f$ are path-connected, and (d) all trajectories asymptotically converge
to a global minimum of $f$.
\end{theorem}
\textbf{}
\begin{proof}
The proof follows the exact same line of logic as the proof to Prop.~\ref{prop:SemiGrad}. The result of Theorem 15.0.3 of \cite{Wiggins2003}, which guarantees that any $\omega$-limit point of gradient descent \eqref{eq:grad_auton} is an equilibrium point, generalizes immediately to the case of natural gradient descent \eqref{eq:nat_grad_auton}.
\end{proof}

\begin{remark}
\label{rem:bad_saddles}
The topology of global optimizers satisfying this semi-contraction condition is the same as those observed when training over-parameterized  networks~\cite{cooper2018loss,liu2020theory}. However, empirical loss functions in these networks often also experience multiple saddle points 
\cite{dauphin2014identifying,jin2017escape}. The attractor sets associated with strict saddles have measure zero \cite{lee2016gradient,lee2017first} under discrete gradient descent with sufficiently small stepsize (i.e., with adequately close approximation to the continuous time case), while the dimensionality of the attractor sets can be further reduced via smoothed versions of the gradient \cite{kreusser2019deterministic}. 

While the presence of  strict saddles precludes the ability of a gradient system to be globally semi-contracting, any of the results given here can be generalized to forward invariant contraction or semi-contraction regions~\cite{Slotine98}. In principle, saddles could then be treated by excluding their measure zero attractor sets from suitably chosen contraction or semi-contraction regions. 
\end{remark}

The topology of equilibria in semi-contracting gradient systems
immediately implies the following result.

\begin{corollary}\label{think_locally}
Consider an autonomous, semi-contracting natural gradient system. If the linearization at some equilibrium point is strictly stable, then all system trajectories tend to this global minimizer.

More generally, if some equilibrium is locally asymptotically stable, all trajectories tend to this global minimizer.

\end{corollary}

\begin{proof} We prove the second part, the first
then follows directly from Lyapunov's linearization method.
Existence of an equilibrium implies existence of a bounded trajectory.
Furthermore, by definition, there exists a ball around the equilibrium point $\x^\star$  such that all trajectories initiated in that ball tend to $\x^\star$. If there was another equilibrium, the path
connecting it to $\x^\star$ would intersect that ball, which
is a contradiction since the path is itself composed of equilibria via Thm.~\ref{thm:natGradNonStrict}.
\end{proof}

\begin{remark}
\label{rem:JacobianOfSemi}
Strict stability of a natural gradient system at an equilibrium point
can of course be established simply by ensuring that all eigenvalues of its Jacobian at this point are strictly in the left-half complex plane.
This condition is equivalent to requiring that the Hessian of the objective function is positive definite at $\x^\star$. 

Indeed, given the natural gradient dynamics \eqref{eq:nat_grad_auton} with $\h(\x) = -\M(\x)^{-1} \nabla f(\x)$, the Jacobian at any equilibrium $\x^\star$ is
\begin{align*}
\left.\frac{\partial \h }{\partial \x}\right|_{\x^\star} &= \left.- \frac{\partial \left[\M^{-1}\right]}{\partial \x}\right|_{\x^\star} \cdot \nabla f(\x^\star)  - \M(\x^\star)^{-1} \nabla^2 f(\x^\star) \\
&=-\M(\x^\star)^{-1} \nabla^2 f(\x^\star)
\end{align*}
Applying a similarity transformation with the symmetric square root of $\M(\x^\star)$ yields
\begin{align*}
&\M^{\frac{1}{2}}(\x^\star) \left.\frac{\partial \h}{\partial \x}\right|_{\x^\star} \M^{-\frac{1}{2}}(\x^\star) \\ &~~~=-\M(\x^\star)^{-\frac{1}{2}} \left[ \nabla^2 f(\x^\star) \right] \M(\x^\star)^{-\frac{1}{2}}
\end{align*}
All eigenvalues of the symmetric matrix above are real, and they are all strictly negative if and only if the Hessian $\nabla^2 f(\x^\star)$ is positive definite. 

Note that this condition is equivalent to the geodesic Hessian at $\x^\star$ being positive definite in any metric, as the Euclidean Hessian is numerically equal to the geodesic Hessian in any metric in this case, due to all terms multiplying the Christoffel symbols in \eqref{eq:riemhessian} being zero.
\end{remark}

\begin{corollary} Consider an autonomous semi-contracting natural gradient system, and assume that the system has more than one equilibrium. Then, at any equilibrium, both the  Jacobian matrix of the dynamics and the Hessian of the objective have at least one zero eigenvalue.
\end{corollary}

\begin{proof} 
Consider an equilibrium $\x^\star$, and an equilibrium path connecting
it to some other equilibrium. The unit tangent vector at $\x^\star$ along this path is an eigenvector of the Jacobian with eigenvalue zero.  Given the algebraic relation between the Jacobian and the objective Hessian pointed out in Remark~\ref{rem:JacobianOfSemi},
this shows in turn that the objective Hessian has a zero eigenvalue.
\end{proof}

\subsection{Examples}

Let us illustrate Theorem~\ref{th_main} using the classical nonconvex Rosenbrock function:
\begin{equation}
f(\x) = 100 (x_1^2- x_2)^2 + (x_1-1)^2 
\label{eq:rosenbrock}
\end{equation}
This function has a unique global optimum at $\x^*=[1,1]\T$, which is located along a long, shallow, parabolic-shaped valley. 

\begin{example}
Consider the Rosenbrock function \eqref{eq:rosenbrock} and the metric \cite{udriste}
\[
\M(\x) = \begin{bmatrix} 400 x_1^2 + 1& -200 x_1 \\
        -200 x_1 &      100
\end{bmatrix} 
\]
The metric $\M(\x)$ satisfies ${\rm tr}(\M(\x))= 400 x_1^2 + 101>0$ and ${\rm det}(\M(\x)) = 100>0$, and thus $\M(\x)\succ0$. Note that $\M(\x)$ is not the Hessian of $f(\x)$. The natural gradient dynamics follows
\[
\dot{\x} = \h(\x)= - \M(\x)^{-1} \grad{f}{}(\x) = -2 \begin{bmatrix} x_1 -1 \\ x_1^2 - 2 x_1 +  x_2 \end{bmatrix}
\]
It can be verified algebraically that 
\[
\M \left( \Jac{\h}{\x} \right) + \left(\Jac{\h}{\x}\right)\T \M + \dot{\M} = -4 \M
\]
which shows that natural gradient descent is contracting with rate $\alpha=2$. This implies that the natural gradient dynamics satisfy
\[
d_{\mathcal{M}}(\x(t), \x^*) \ \le \ {\rm e}^{-2t} \ d_{\mathcal{M}}(\x(0), \x^*)
\]
where $\x^* = [1,1]\T$. Equivalently, the Rosenbrock function is geodesically $\alpha$-strongly convex with $\alpha = 2$. 
\end{example}

The Rosenbrock metric $\M(\x)$ can be viewed as following from a differential change of variables
\[
\delt{z} = \boldsymbol{\Theta}(\x) \delt{x} = \begin{bmatrix} 20 x_1 & -10 \\ 1 & 0 \end{bmatrix} \delt{x}
\]
where $\M=\boldsymbol{\Theta}\T\boldsymbol{\Theta}$ yields $\ \delt{x}\T\M\delt{x} = \delt{z}\T \delt{z}$. This differential change of variables is integrable, so that g-convexity of the Rosenbrock can be shown using the explicit nonlinear coordinate change $z_1 = 10 x^2_1 - 10 x_2$ and $z_2 = x_1-1$ that provides $f = z_1^2 + z_2^2$.

\begin{example} \label{psi} Mirror descent provides another example of a metric corresponding to an explicit state transformation, with Newton's method as a special case. 

Consider a twice differentiable scalar objective function $f(\x)$, and a smooth strictly convex scalar function $\psi(\x)$. Denoting by $\H_f(\x) = \nabla^2 f(\x)$ and $\H_\psi = \nabla^2 \psi$ the Hessians of these functions, continuous-time mirror descent of $f(\x)$ under potential $\psi(\x)$ corresponds to natural gradient in the Hessian metric $\H_\psi$  \cite[Sec.~2.3]{gunasekar2018characterizing}
\begin{equation}
\dot{\x} = - \H_\psi^{-1} \nabla f(\x)
\label{eq:mirror_in_x}
\end{equation}
Consider the explicit change of variables $\z = \nabla \psi(\x)$, which can be written in differential form as $\ \delta \z = \H_\psi \delta \x$. 
The dynamics \eqref{eq:mirror_in_x} can be viewed in the mirror space as 
\begin{equation}\nonumber
\dot{\z} = \H_\psi \, \dot{\x} = - \nabla f(\x)
\label{eq:mirror_in_z}
\end{equation}
and therefore
\begin{equation}\nonumber
\ddt \ \delta \z = - \H_f \delta \x
\end{equation}
Letting  $\M(\x) = \H_\psi^2$, this yields
\begin{equation}\nonumber
\ddt \left[\delta \x^T \M(\x) \delta \x \right] = 
\ddt \left[ \delta \z^T \delta \z \right] 
= - 2 \ \delta \z\T \H_f \ \delta \x 
 = - \delta \x\T \left[ \H_f \H_\psi + \H_\psi \H_f \right] \delta \x\T 
\end{equation}
Thus, continuous mirror descent \eqref{eq:mirror_in_x} is contracting with rate $\lambda > 0$ in the metric $\M(\x) = \H_\psi^2$ if
\begin{equation}
\H_f \H_\psi + \H_\psi \H_f \succeq 2 \lambda \H_\psi^2
\label{eq:contractionConditionMirror}
\end{equation}

In the particular case when $f$ is $\alpha$-strongly convex and the potential function is chosen as $\psi(\x) = f(\x)$, equation \eqref{eq:mirror_in_x} simply corresponds to Newton's method, and \eqref{eq:contractionConditionMirror} verifies that Newton's method is contracting with rate 1 in the squared Hessian metric $\M(\x) = \H_f^2(\x)$ \cite{lohmiller2009exact}.

Note that the well-known result that the transformation $\z = \grad{\phi}{}(\x)$ is one-to-one (given the  strict convexity of $\psi$) can also be shown by constructing, for a given $\z$, the system
\begin{equation}
\dot{\x} + \grad{\psi}{}(\x) = \z \label{eq:InvertGrandientViaDynamics}
\end{equation}
which is autonomous and contracting in the identity metric and thus must reach a unique equilibrium point.

\end{example}

The following proposition provides further insight into the case when the contraction metric is related to an explicit change of variables more generally.

\begin{proposition}[Relationship between gradient and natural gradient under a diffeomorphic change of variables]\label{change_variable}
Consider a diffeomorphic change of variables $\z = \g(\x)$, and the
associated metric $\M(\x) =\boldsymbol{\Theta}(\x)\T \boldsymbol{\Theta}(\x)$, with $\boldsymbol{\Theta}(\x) = \jac{\g}{\x}$. For any twice differentiable function $f : \Rn \rightarrow \R$, natural gradient descent in $\x$ 
\[
\dot{\x} = -\M(\x)^{-1} \grad{f}{}(\x)
\]
is equivalent to gradient descent in $\z$ 
\[
\dot{\z} = - \grad{\left[f \circ \g^{-1}\right]}{}(\z)
\]
\end{proposition}
\begin{proof} In the $\z$ coordinates we have
\begin{align*}
\ \ \dot{\z} = \jac{\g}{\x} \dot{\x} \ 
= \ - \boldsymbol{\Theta} \ \M^{-1} \grad{f}{}(\x) \ 
& = \ - \left[ \jac{f}{\x} \boldsymbol{\Theta}^{-1} \right]\T \\ 
 &= \ - \left[ \jac{ f \circ \g^{-1}}{\z} \right]\T \ 
= \ - \grad{[f \circ \g^{-1}]}{}(\z)
\end{align*}
\end{proof}

\begin{proposition}
\label{prop:curvature}
Consider a metric $\M(\x)$ and suppose there exists a diffeomorphic change of variables  $\z = \g(\x)$ such that $\M(\x) =\boldsymbol{\Theta}(\x)\T \boldsymbol{\Theta}(\x)$, with $\boldsymbol{\Theta}(\x) = \jac{\g}{\x}$. Then, the associated Riemannian curvature tensor with components $R_{ik\ell m}$
must be identically zero.
\end{proposition}

\begin{proof}
Note that since $\delta \z = \boldsymbol{\Theta}(\x) \delta \x$, it follows that $\delta \z\T \delta \z = \delta \x\T \M(\x) \delta \x$ and thus the Riemannian metric tensor expressed in the $\z$ coordinates is the identity. Since the components of the Riemannian metric tensor are constant in these transformed coordinates, it follows that the components of the Riemannian curvature tensor are identically zero \cite{udriste}. Transformation laws for tensors ensure that the components of the curvature tensor remain zero under arbitrary coordinate change, thus $R_{ik\ell m}=0$.  
\end{proof}

The general freedom to consider differential changes of coordinates $\delta \z = \boldsymbol{\Theta}(\x) \delta \x$ where $\boldsymbol{\Theta}$ is non-integrable provides additional flexibility and generality to both contraction analysis and g-convexity, as illustrated by the following examples.

\begin{example}
\label{ex:DescentContracting}
Consider the non-convex function
\[
f(\x) = x_1^2+x_2^2+x_1^2 x_2^2
\]
which has a global minimum at $\x=\mathbf{0}$. Contours of the function are shown in Fig.~\ref{fig:DescentContracting}. Gradient descent 
\[
\dot{\x} = - \nabla f(\x) = -2 \begin{bmatrix} x_1 ( 1+x_2^2) \\ x_2 (1+x_1^2) \end{bmatrix} 
\]
can be shown to be contracting at rate $\lambda=2$ in the metric
\[
\M(\x) = \begin{bmatrix} 2+x_1^2 & \ -x_1 x_2 \\ -x_1 x_2 & \  2+x_2^2 \end{bmatrix}
\]
Fig.~\ref{fig:DescentContracting} shows two solutions and plots their geodesic distance. The decay is, as expected, at a rate faster than the exponentially decreasing upper bound as derived from \eqref{d_M}. The curvature tensor for this metric has some non-zero components,
such as
\[R_{1221} = \frac{2}{2+x_1^2+x_2^2}
\]
From Proposition \ref{prop:curvature}, this shows that this metric cannot be derived from an explicit change of coordinates.
\end{example}

\begin{figure}
    \centering
    \includegraphics[width=\columnwidth]{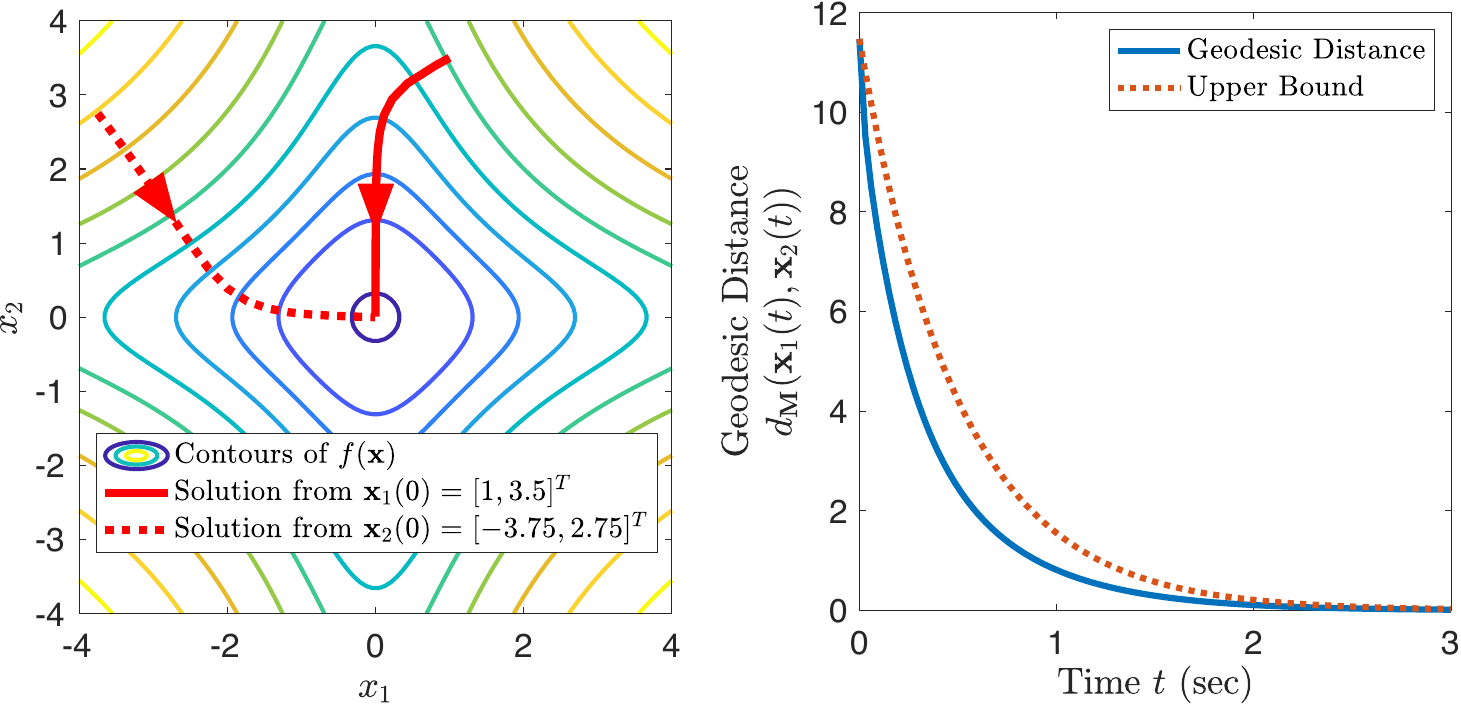}
    \caption{Contracting gradient descent corresponding to Example \ref{ex:DescentContracting}.}
    \label{fig:DescentContracting}
\end{figure}

\begin{example}
\label{ex:semiDescent}
Consider the function
\[
f(\z) =z_1^8 - z_1^6 - 2 z_1^5 z_2 - z_1^4 + z_1^2 z_2^2 + z_1^2 + 2 z_1 z_2 + z_2^2
\]
and natural gradient descent with a given natural metric $ \boldsymbol{\Theta}(\z)\T \boldsymbol{\Theta}(\z)$, where 
\[
\boldsymbol{\Theta}(\z) = \begin{bmatrix} 1 & 0 \\
1 -3 z_1^2 & 1 \end{bmatrix}
\]
This natural gradient dynamics is verified semi-contracting in the metric
\[
\M(\z) = \boldsymbol{\Theta}(\z)\T \begin{bmatrix} 1 + z_2^2 - 2 z_1^3 z_2 + z_1^6 & 0 \\
0 & 1 + z_1^2 \end{bmatrix} \boldsymbol{\Theta}(z)
\]
Similar to Example~\ref{ex:DescentContracting}, this metric has non-zero Riemannian curvature, and thus cannot be derived from a change of coordinates. Figure~\ref{fig:SemiDescent} shows the contours of $f$ and two solutions of natural gradient descent. Figure \ref{fig:SemiDescentFunction} shows that the the geodesic distance between these two solutions is non-increasing. Since the system is only semi-contracting, the distance between solutions does not tend toward zero. It can be verified that $f$ is  a sum of squares and thus $f(\z)\ge0$, and that $f(\z)=0$ when $z_2 = z_1^3-z_1$. Both initial conditions asymptotically lead to this path connected set of global optima.
\end{example}

\begin{figure}[t]
    \centering
    \includegraphics[width=\columnwidth]{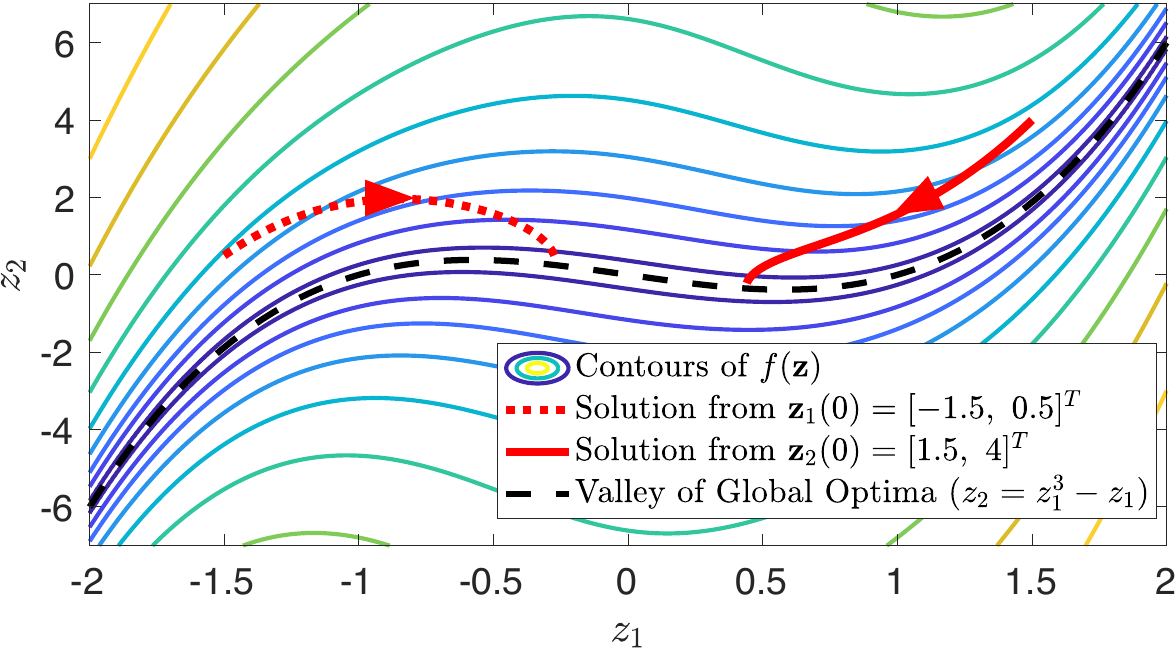}
    \caption{Semi-Contracting Natural Gradient Descent for Example~\ref{ex:semiDescent}. }
    \label{fig:SemiDescent}
\end{figure}

\begin{figure}[t]
    \centering
    \includegraphics[width=\columnwidth]{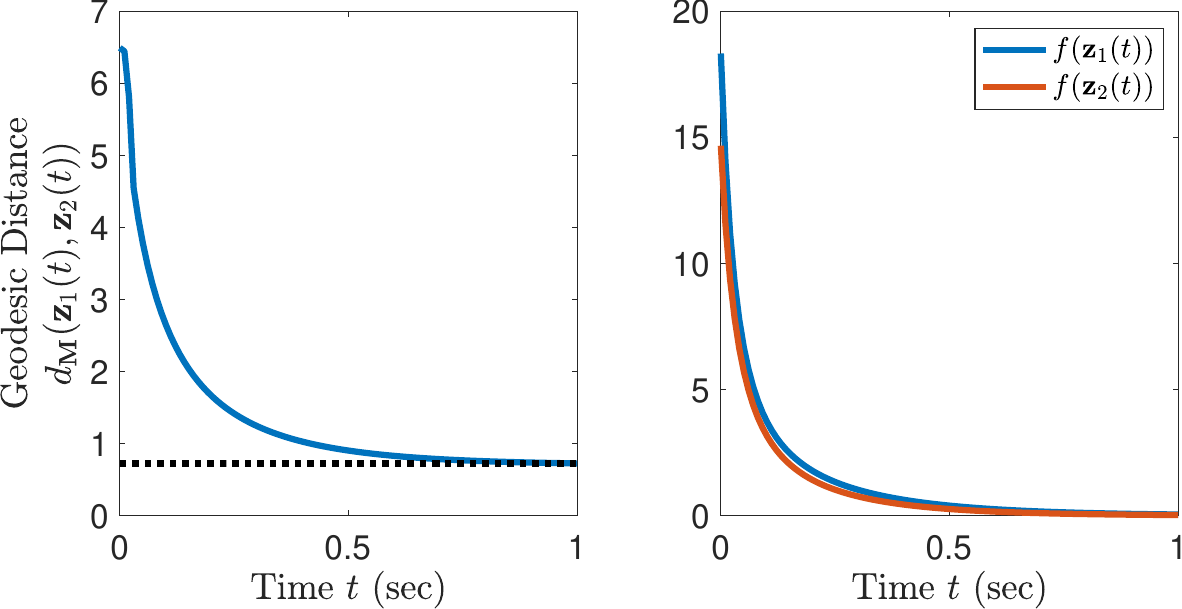}
    \caption{Semi-Contracting Natural Gradient Descent for Example~\ref{ex:semiDescent}. }
    \label{fig:SemiDescentFunction}
\end{figure}

\begin{example}
Geodesically-convex optimization can also be used to carry out manifold-constrained optimization in an unconstrained fashion via recasting problems over a Riemannian manifold directly \cite{Absil08,SRA_Conic_Optimization}. Taking an intrinsic view of the manifold, coordinate free results are available \cite{bullo}, however, for the purposes of computation, we assume a global coordinate chart here. Consider for instance optimization over the set $\mathbb{S}^n_+$ of $n\times n$ positive definite matrices, and specifically the problem of finding the Karcher mean of $m$ matrices $\A_i \in \mathbb{S}^n_+$ \cite{Sra_First_Order_Methods}, which minimizes the objective function 
\[
f(\X) = \frac{1}{2}\sum_{i=1}^m\| {\rm log}(\A_i^{-1}\,\X) \|_F^2
\]
where $\|\A\|_F=\sqrt{{\rm tr}(\A\T \A)}$ denotes the Frobenius norm of a matrix $ \A$. The function $f(\X)$ is $m$-strongly convex~\cite{Sra_First_Order_Methods} on $\mathbb{S}^n_+ \ $in the metric that measures symmetric differential displacements as
\begin{equation}
ds^2 = {\rm tr}\left( \left( \delta \X \ \X^{-1} \right)^2 \right)  
\label{eq:PSD_Riem}
\end{equation}
Naturally, the requirement that $\delta \X$ be symmetric makes it an element of the tangent space to the manifold of symmetric positive definite matrices.  

This metric generalizes the GP case \eqref{eq:GP_Gconvex}, and coincides with the second-order terms in the Taylor series of the log barrier $-{\rm logdet}(\X)$ \cite{BoydVanberghe04}. The gradient of $f(\X)$ can be written
\[
\grad{f}{}(\X) = \sum_{i=1}^m {\rm log}(\A_i^{-1}\,\X )\,  \X^{-1}
\]
and accordingly the natural gradient can be shown to satisfy
\[
\sum_{i=1}^m \ \X \, {\rm log}(\A_i^{-1}\,\X) = \X \, \grad{f}{}(\X) \, \X
\]
From Theorem \ref{th_main},  any trajectory with arbitrary initial condition $\X(0) \in \mathbb{S}^n_+$ will remain within $\mathbb{S}^n_+$ under the natural gradient descent dynamics
\[
\dot{\X} = - \sum_{i=1}^m  \ \X \, {\rm log}(\A_i^{-1}\,\X)
\]
since (intuitively) the Riemannian metric \eqref{eq:PSD_Riem} makes any element on the boundary of the positive definite cone an infinite distance away from any one in the interior, and contraction of the natural gradient dynamics ensures that geodesic distances decrease exponentially.
\end{example}

\begin{example}
\label{ex:MatBreg}
An approximation to the Riemannian distance of two positive definite (PD) matrices on the PD cone is given by the Bregman LogDet divergence on $\mathbb{S}_+^n$
\begin{equation}
    \label{eq:breg}
d(\A||\X) = {\rm logdet}(\A^{-1} \X) + {\rm tr}(\X^{-1} \A) - n
\end{equation}
The metric is convex in its first argument, and can be shown to be geodesically convex in the second. We illustrate the connection with contraction to show this property. Note that
\[
\grad{d(\A||\X)}{\X} = \X^{-1} - \X^{-1} \A \X^{-1} 
\]
so that the natural gradient descent dynamics are simply
\[
\dot{\X} = - \X \left[ \grad{d(\A||\X)}{\X} \right] \X = -\X + \A
\]
with differential dynamics 
\[
\delta \dot{\X} = - \delta \X
\]
where the differential displacement $\delta \X$ must be symmetric.
Considering the rate of change in length of these differential displacements
\[
\ddt {\rm tr}\left( (\X^{-1} \, \delta \X)^2 ) \right) = -2{\rm tr}(\X^{-1} \A (\X^{-1} \delta \X^{-1})^2 )
\]
and defining the differential change of variables $\delta \Z = \X^{-\frac{1}{2}}\, \delta \X \, \X^{-\frac{1}{2}}$, one has $\tr{ \delta \Z^2} = {\rm tr}\left( (\X^{-1} \, \delta \X)^2  \right) $ and
\begin{align}
\ddt \tr{ \delta \Z^2} = -2 \tr{ \delta \Z \, \X^{-\frac{1}{2}} \A \X^{-\frac{1}{2}} \,\delta \Z}   = -2 \| \A^{\frac{1}{2}} \X^{-\frac{1}{2}} \,\delta \Z\|_F^2 <0
\end{align}
for all $\delta \Z\ne \mathbf{0}$. Hence, considering only the second argument to LogDet divergence, its Riemannian Hessian is positive definite, thus proving g-convexity via Thm.~\ref{th_main}.
\end{example}

\subsection{Non-autonomous Systems and Virtual Systems}

In our optimization context, the fact that contraction analysis is directly applicable to non-autonomous systems can be exploited in a variety of ways. As we shall detail later, a key aspect is that it allows feedback combinations or hierarchies of contracting modules
to be exploited to address more elaborate optimization problems or architectures. Also, it makes the construction of \emph{virtual} systems~\cite{wei} possible to potentially extend results beyond natural-gradient descent.

\begin{remark}  The natural gradient $\ \M^{-1}(\x)\ \grad{f}{\x}(\x,t)$ represents the direction of steepest ascent on the manifold at any given time. With this in mind, Remark~\ref{rem:robust} on robustness enables convergence analysis for natural gradient descent within time-varying optimization contexts~\cite{Kostya}. Let $\x^*(t)$ denote the optimum of a time-varying $\alpha$-strongly g-convex function. If $\sqrt{\dot{\x}^*(t)\T \M(\x) \dot{\x}^*(t)}<R$, then the natural gradient will track $\dot{\x}^*(t)$ with accuracy $R/\alpha$ after exponential transient.
\end{remark}

\begin{remark}
Consider a contracting natural gradient system of the form \eqref{eq:NatGrad}. In the autonomous case, equations governing the differential displacement follow
\begin{equation}\label{delta_x}
\ddt \ \delta \x \ = \ \Jac{\h}{\x} \ \delta \x 
\end{equation}
which has a similar structure to the time evolution of $\h(\x)$
\begin{equation}\label{ddt_h}
\ddt \ \h(\x) \ = \ \left( \grad{ \h}{}(\x) \right) \ \h(\x)
\end{equation}
Thus, for natural gradient descent of an $\alpha$-strong g-convex function $f(\x)$, the same algebra leading to~\eqref{ddt_z^2} also gives 
\[
\ddt (\h\T \M \h) = -2 \h\T \H \h \le - 2 \alpha (\h\T \M \h)
\]
so that the Krasovskii-like function
\[
V(\x) = \h(\x)\T \M(\x) \h(\x) = \grad{f}{}(\x)\T \M(\x)^{-1} \grad{f}{}(\x)
\]
can be viewed as an exponentially converging
Lyapunov function, with global minimum $V=0$ at the unique 
minimum of $f(\x)$. Of course, \eqref{delta_x} remains
valid for \emph{non-autonomous} systems as well, while~\eqref{ddt_h} does not.
\end{remark}


The use of virtual contracting systems~\cite{wei,jouffroy,manchester2017CCCM} allows guaranteed exponential convergence to a unique minimum to be extended to classes of dynamics which are not pure natural gradient. For instance, it is common in optimization to adjust
the learning rate as the descent progresses.
Consider a natural gradient descent with the function $f(\x)$ $\alpha$-strongly g-convex in metric $\M(\x)$, and define the new system
\begin{equation}
\dot{\x} = - p(\x,t)\, \M(\x)^{-1}\, \grad{f}{}(\x)
\label{p(x,t)}
\end{equation}
where the smooth scalar function $p(\x,t)$ modulates the learning rate~\cite{Amari98b} and is uniformly positive definite,
\[ \exists \ p_{min} > 0 , \ \forall t \ge 0, \ \forall \x, \ \ \ \
p(\x,t) \ge p_{min} 
\]
Let us show that this system tends exponentially to the minimum $\x^*$ of $f(\x)$.

Consider the auxiliary, \emph{virtual} system,
\begin{equation}
\dot{\y} = - \M(\y)^{-1}\, \grad{(\ p(\x,t)f(\y)\ )}{\y}
\label{virtual_p(x,t)}
\end{equation}
For this system, $p(\x(t) ,t)$ is an external, uniformly positive definite function of time, and thus 
\[\grad{(\ p(\x,t)f(\y)\ )}{\y} \ = \ p(\x,t) \ \grad{f}{}(\y)
\]
so that the contraction of~\eqref{eq:NatGrad} with rate $\alpha$ implies the contraction of~\eqref{virtual_p(x,t)} with rate $ \ \alpha p_{min} \ $. Since both $\x(t)$ and $\x^*$ are particular solutions
of~\eqref{virtual_p(x,t)}, this implies in turn that 
$\x(t)$ tends to $\x^*$ with rate $ \ \alpha p_{min} \ $.


Note that since we only assumed that $p(\x,t)$ is uniformly positive definite, in general the actual system~\eqref{p(x,t)} is not contracting with respect to the metric $\M(\x)$. 

\begin{remark}
The  learning rate may also be selected to improve the numerical properties of the algorithm in a discrete time implementation. For example, $p(\x,t)$ could vary as the inverse of the condition number of $\nabla^2 f(\x)$ to improve numeric conditioning without impact on stability guarantees.
\end{remark}

\subsection{Contraction $+$}

Corollary~\ref{think_locally} above points to a more general class of results where contraction or semi-contraction properties are combined with other information, such as a stable local linearization or a decreasing cost, to provide global results.

\subsubsection{Contraction is attractive} 

As we now show, Corollary~\ref{think_locally} extends more generally to autonomous semi-contracting systems. An instance of this result in the case of an identity metric was derived in~\cite{bulloSemi2020}.

\renewcommand{\bl}{\beta_0}
\renewcommand{\bu}{\beta_1}
\newcommand{\I}{\mathbf{I}}

\begin{proposition} \label{7}
Consider an autonomous system
\begin{equation}
\dot{\x} = \h(\x)
\label{eq:autonomousSystem}
\end{equation}
semi-contracting in a bounded metric $\M(\x)$,
\[
{\bf 0} \prec \bl^2 \I \preceq \M(\x) \preceq \bu^2 \I \quad \quad \forall \x
\] 
If a system equilibrium is locally asymptotically stable, then it is
globally asymptotically stable. In particular, if the system linearization at some equilibrium point is strictly stable, then all system trajectories tend to this equilibrium.
\end{proposition}

\begin{proof} 
The result is a particular case of Theorem~\ref{3}, to be
discussed next.
\end{proof}

\begin{theorem}\label{3}
Consider a non-autonomous system, semi-contracting in a bounded metric $\M(\x)$,
\begin{equation}
\label{eq:metricConditions}
{\bf 0} \prec \bl^2 \I \preceq \M(\x) \preceq \bu^2 \I \quad \quad \forall \x
\end{equation}
Assume that a \emph{specific} trajectory $\x^\star(t)$ is locally attractive. Then all trajectories tend asymptotically to $\x^\star(t)$.

In particular, if contraction holds (possibly in a different bounded metric) along a specific trajectory $\x^\star(t)$, and within a tube of constant size around it, then all trajectories tend asymptotically to $\x^\star(t)$.
\end{theorem}

\newcommand{\B}{\mathcal{B}}
\newcommand{\xl}{\underline{\x}}
\begin{proof} 
The first part generalizes the equilibrium argument from \cite{bulloSemi2020} to arbitrary trajectories and arbitrary metrics.
Assume that $\x^*(t)$ is locally attractive, by which we mean there exists some $\epsilon>0$ such that, for any initial time $t_0$ and initial condition $\x_0 \in \B_{
\mathbf{I}}( \x^*(t_0), \epsilon ) $, one has $\x(t_0+T; \x_0, t_0) \rightarrow \x^*(t_0+T)$ as $T\rightarrow +\infty$. Without loss of generality, we assume $t_0= 0$. 

\newcommand{\dI}{d_{\mathbf{I}}}

Consider some generic initial condition $\x_0$ with $\dI(\x_0, \x^*(0)) > \epsilon$. We will argue that there is always a finite time window over which the geodesic distance from $\x(t;\x_0)$ to $\x^*(t)$ decreases by a fixed finite increment. 

Consider a geodesic connecting $\x_0$ and $\x^*(0)$ and denote by $\xl$ the unique point on this geodesic that is a geodesic distance $\bl \epsilon$ away from $\x^*(0)$. Due to the uniform positive definiteness of $\M$, this condition implies that $\xl \in \B_{\mathbf{I}}(\x^*(0),\epsilon)$.

Because of the local attractivity of $\x^*(t)$, there exists a time $t_1>0$ such that
\[
d_\I(\x(t_1 ; \xl) , \x^*(t_1) ) \le \frac{\epsilon \bl}{2 \bu}
\]
which further implies that
\[
d_\M(\x(t_1;\xl),  \x^*(t_1) ) \le \bl \frac{\epsilon}{2}
\]
In addition, since the system is semi-contracting, we have
\[
d_\M(\x(t_1; \xl) , \x(t_1; \x_0) ) \le d_\M(  \xl , \x_0 )
\]
and so by the triangle inequality
\begin{align*}
d_\M( \x^*(t_1) , \x(t_1;\x_0) ) \le d_\M(\xl, \x_0) + \bl \frac{\epsilon}{2}
&= d_\M(\xl, \x_0)  + d_\M(\x^*(0), \xl) \\
&\phantom{=}~~~~~~~~~~~ - d_\M(\x^*(0), \xl) + \bl \frac{\epsilon}{2} \\
&= d_\M(\x^*(0),\x_0) - \bl \frac{\epsilon}{2} 
\end{align*}
This implies that so long as $\dI( \x_0, \x^*(0)) > \epsilon$, the trajectory from $\x_0$ will eventually decrease its geodesic distance by a fixed finite increment. Since this process can be repeated, it follows that there must exist some time $T$ such that $\x(T; \x_0) \in \B_{\mathbf{I}}(\x^*(T),\epsilon)$.

To complete the second part of the proof we proceed to show that if contraction \eqref{eq:contraction_condition} holds within a tube of constant size around trajectory $\x^*(t)$,  for some bounded metric $(\beta_0^\star)^2 \I \preceq \M^\star(\x) \preceq (\beta_1^\star)^2 \mathbf{I} \ $ and some rate $\alpha^\star>0$, then that trajectory is locally attractive. By  condition \eqref{eq:contraction_condition}  holding within a tube we mean that there exists some $\epsilon >0$ such that \eqref{eq:contraction_condition} holds for any time $t$ and any $\x \in \B_\I(\x^*(t),\epsilon)$. From boundedness of the metric, we have
\[
\beta_0^\star \, d_{\I}(\x^*(t), \x) \le d_{\M^\star}( \x^*(t), \x) 
\]
so that any initial condition $\x_0$ satisfying
\[
d_{\M^*}(\x^*(0), \x_0) \le \epsilon \beta_0^*  
\]
necessarily starts within this tube. Further, since \eqref{eq:contraction_condition} holds within the tube, it follows that the geodesic ball of radius $\epsilon \beta_0^\star$ around $\x^\star(t)$ is forward invariant. Since this ball is contained within a contraction region, this implies that for any $\x_0 \in \B_\M(\x^\star(0), \beta_0^\star \epsilon)$
\[
d_{\M^*}(\x^*(t), \x(t;\x_0 ) ) \le e^{-\alpha^\star t} \  d_{\M^*}(\x^*(0), \x_0 ) 
\]
which proves local asymptotic stability of $\x^\star(t)$.
\end{proof}


\begin{remark}
The condition regarding contraction within a tube of fixed size is included to avoid pathological cases where the region of contraction shrinks to zero as $t\rightarrow +\infty$. For example, the system $\dot{x} = -x + t x^3$ is contracting with rate $1$ at the origin for all time,  yet the origin is not locally asymptotically stable. 
\end{remark}

\begin{remark}
Intuitively, the result can be understood by analogy with a shrinking rope. Consider a path of initial conditions connecting $\x^\star(0)$ to any $\x_0$. As this path flows forward in time, at $t=0$, only a portion of this path of states is within the basin of attraction for $\x^\star(t)$. Viewing this path as a rope, the semi-contraction property ensures that no part of the rope can increase in length as it flows forward through the dynamics. Yet, due to local attractivity at one end of the rope, a portion of it is guaranteed to have shrinking length, pulling the rest of the rope toward the the region of attraction.  
\end{remark}

\begin{remark}
Numerical tools for determining contraction metrics \cite{aylward2008stability} are based on the fact that contraction conditions \eqref{eq:contraction_condition} are convex in the metric {\em for a fixed contraction rate}. In practice, these methods often involve an outer search procedure for the contraction rate (e.g., via a binary search). In this sense, the use of semi-contraction is desirable as it does not require this additional search.  
\end{remark}

\begin{remark}
These results have analogs in the context of controller design using 
control contraction metrics (CCMs)~\cite{manchester2017CCCM,singh2017robust}. In this setting, one can impose a semi-contracting closed-loop metric everywhere, except in a tube along a desired trajectory where a strict contraction condition would be required, possibly in a different metric. Since the existence of an exponential (resp. semi) CCM implies that the closed-loop plant can be rendered contracting (resp. semi-contracting), Theorem~\ref{3} would then imply asymptotic stabilizability of the desired trajectory.

This extension likewise has analogs for manifold convergence results \cite[Section 5]{manchester2017CCCM} and  convergence to a limit cycle by transverse contraction~\cite{Manchester14_transverse}, both of which are special cases of CCM results applied to suitably constructed virtual control systems~\cite{manchester2017CCCM}. In either case, a semi-contracting CCM everywhere can be combined with a contracting CCM condition on the manifold (or limit cycle) and within a neighborhood of it to assert asymptotic stability of the manifold (or limit cycle). In the limit cycle case for autonomous systems, the contracting CCM condition needs only be enforced on the limit cycle itself, as its satisfaction within some neighborhood is then guaranteed by compactness. Likewise, for convergence to a compact manifold (e.g., an eggshell) in an autonomous system, the contracting CCM condition needs only be considered on the manifold itself.
\end{remark}

\subsubsection{Contraction as minimization}

Similarly, Proposition \ref{prop:contracting_grad_desc} may be viewed as a particular instance of the following results, which use contraction properties to minimize a cost or Lyapunov-like function.

\begin{proposition}[Exponential Cost Minimization]\label{8}
Consider an autonomous contracting system \eqref{eq:autonomousSystem}, and a scalar cost function
$V(\x)$ such that $\dot{V}(\x) \le 0 $ for all $\x$. Then all 
trajectories tend exponentially to a global minimum of $V$.
\end{proposition}
\begin{proof} Because the system is contracting and autonomous, it
tends exponentially to a unique equilibrium $\x^\star$ \cite{Slotine98}. Consider
now an arbitrary $\x$, and the  system trajectory initialized at $\x$. Since the cost $V$ can only decrease along the trajectory, this implies that  $V(\x^\star) \le V(\x)$, for all $\x$.
\end{proof}

\begin{proposition}\label{9}
Consider an autonomous semi-contracting system \eqref{eq:autonomousSystem} in a bounded metric $\M(\x)$, and a scalar cost function $V(\x)$ such that $\dot{V}(\x) \le 0 $ for all $\x$. Assume that one system equilibrium $\x^\star$ is locally attractive (e.g., that linearization at $\x^\star$ is strictly stable). Then this equilibrium is unique, it is a global minimum of $V$, and all trajectories converge to it asymptotically. 
\end{proposition}

\begin{proof}
Applying Proposition~\ref{7} shows that all trajectories asymptotically
tend to $\x^\star$, which also implies that the equilibrium is unique.  By the same reasoning as in Proposition~\ref{8}, since $V$ can only decrease, $V(\x^\star)$ must be a global minimum.
\end{proof}

\begin{remark}\label{16}
These results extend readily to the case where a system is semi-contracting within some forward invariant region, as opposed to globally. These generalizations may have applicability e.g., to the continuous-time limit of trained neural networks~\cite{deep_residual,neural_ode,dupont}, wherein semi-contraction regions represent basins of attraction that are free of saddles. Metrics may become singular as they approach the boundary of these open sets \cite[Section 3.9]{Slotine98}, allowing the semi-contraction region to cover the entire basin. 
\end{remark}

Proposition~\ref{9} can be stated more generally as follows.

\begin{theorem}[Asymptotic Cost Minimization]\label{10}
Consider an autonomous semi-contracting system in a bounded metric $\M(\x)$, and a scalar cost function $V(\x)$ such that $\dot{V}(\x) \le 0 $ for all $\x$. Assume that one trajectory is known to be bounded.
Let $\script{I}$ be a forward invariant set where $\dot{V}=0$, and assume that the contraction condition \eqref{eq:contraction_condition} holds on $\script{I}$ for some (possibly different) metric. 

Then $\script{I}$  is path connected, all system trajectories converge to a unique equilibrium $\x^\star \in \script{I}$, and $V$ is globally 
minimized at $\x^\star $.
\end{theorem}

\begin{proof}
Let us first show that $\script{I}$ is path connected, by contradiction. 
Assume $\script{I}$ is not path connected, then it can be decomposed into
two disjoints subsets, $\script{I}_1$ and $\script{I}_2$. Because $\script{I}$ is invariant and the subsets are disjoint, each of the subsets must be invariant. Strict contraction on $\script{I}_1$ and $\script{I}_2$ then implies that each subset contains at least one locally stable equilibrium point (note that each of the subsets may themselves be disconnected and thus may contain more than one stable equilibrium point). The existence of two equilibrium points contradicts Proposition~\ref{9}, and thus $\script{I}$ is path connected.

Next, on the connected invariant set $\script{I}$, contraction implies that the geodesic distance between any two points shrinks exponentially. By the same reasoning as in Proposition~\ref{8}, this in turn implies convergence to a global minimum of $V$.
\end{proof}

\begin{remark} 
Note that for a system where a scalar cost $V$ satisfies $\dot{V}\le 0$, radial unboundedness of  $V$ is a sufficient condition for all trajectories to be bounded, ensuring the existence of a bounded trajectory as necessary in Thm.~\ref{10}. 
\end{remark}

\begin{remark}
In the case of mechanical systems, $V$ may often be chosen as the total energy of the system, so that Proposition~\ref{8} implies exponential convergence of the total energy, and, in turn, that potential energy is exponentially minimized. Similarly, Theorem~\ref{10} implies
that potential energy is asymptotically minimized.
\end{remark}

\begin{remark}
Contraction criteria can also be expressed in non-Euclidean norms and their associated matrix measures (\cite{Slotine98}, section 3.7.ii). The results above extend immediately to these representations.
\end{remark}

\section{Primal-Dual Optimization}\label{GPD_opt}
   
Primal-Dual algorithms are widely used
in optimization to determine saddle points
and also appear naturally in constrained optimization~\cite{BoydVanberghe04}, where
Lagrange parameters play the role of dual variables.
When a function is strictly convex in a subset of its variables, and strictly concave in the remaining, gradient descent/ascent dynamics converge to a unique saddle 
equilibrium~\cite{arrow1958studies,paganini}. Within the context of constrained optimization, these dynamics are known as the primal-dual dynamics.
Such dynamics play an important role e.g., in machine learning, for instance in adversarial training~\cite{madry2017towards}, in the information theory~\cite{tishby2015deep} of deep networks, in
reinforcement learning~\cite{Wang17} and actor-critic methods,
and in support vector machine representations~\cite{kosaraju2018primal}. More generally, they are central to a large class of practical min-max problems, such as problems in physics involving
free energy, or, e.g.,
nonlinear electrical networks modeled in terms of Brayton-Moser mixed potentials~\cite{Kostya,ortega2003power, cavanagh2018transient}.

Consider a scalar function $\script{L}({\bf x},{\bf \llambda}, t)$,
possibly time-dependent, and metrics 
$\M_\x(\x)$ and $\M_{\llambda}(\llambda)$.
Consider the natural primal-dual dynamics, which 
we define as
\begin{subequations} \label{eq:PD}
\begin{align} \label{eq:primal}
{\bf M_\x}({\bf x})\ \dot{\bf x} \ &= -\nabla_{\bf x} \script{L}(\x,\llambda,t) \\
{\bf M_\llambda}({\bf \llambda})\ \dot{\llambda} \ &= \ \nabla_{\llambda}  \script{L}(\x,\llambda,t) \label{eq:dual}
\end{align}
\end{subequations}
In contrast to Remark \ref{rem:bad_saddles}, wherein spurious saddle equilibrium points presented an obstacle to global contraction, here the target equilibrium points of these dynamics are, by construction, chosen to be the saddle points of the function $\script{L}$.
Using the metrics ${\bf M_x}({\bf x})$ and ${\bf M_\llambda}({\llambda})$ extends the standard case \cite{Kostya}, where they would be replaced by constant, symmetric positive definite matrices. The practical relevance of this extension is illustrated by the following example in the case of natural adaptive control.

 

\subsection{Primal Dual Dynamics in Natural Adaptive Control}

This section illustrates the presence of natural primal-dual dynamics embedded in the application of natural adaptive control laws. Consider a system given by 
\begin{equation}
    \label{eq:sys_for_adaptive}
\J(\x,\ldots,\x^{(n-2)},\a) \x^{(n)} + \Y(\x,\ldots,\x^{(n-1)}) \a = \u
\end{equation}
with configuration $\x\in \R^N$, control $\u \in \R^N$, and unknown parameters $\a \in \Acal \subset \R^p$. The regressor $\Y \in \R^{N \times p}$ and symmetric matrix $\J \in \R^{N\times N}$ may depend nonlinearly on the state and its derivatives. We assume that the matrix $\J$ remains positive definite for all $\a \in \Acal$ and that it is linear in $\a$. As a result, there exists a regressor function $\W$ such that
\begin{align}
\W(\x, \ldots, \x^{(n-1)}, \x^{(n)}) \a \ =&  \ \J(\x,\ldots,\x^{(n-2)},\a) \x^{(n)} \\&+ \Y(\x,\ldots,\x^{(n-1)}) \a
\end{align}
and a regressor function $\Q$ such that for any $\s\in\R^N$
\[
-\frac{1}{2} \dot{\J} \s = \Q(\x, \ldots, \x^{(n-1)}, \s)\, \a
\]

Consider a desired trajectory $\x_d(t)$ and the associated sliding variable \cite{slotine1987adaptive,slotine1991applied}
\begin{equation}
\s = \left(\ddt + \lambda\right)^{n-1} \xt = \x^{(n-1)} -\x_r^{(n-1)} 
\label{eq:sliding}
\end{equation}
where $\xt = \x - \x_d$. With this sliding variable, we define a reference $\x_r^{(n-1)}$ for the order $n-1$ derivative of the state. 

Choosing the control law
\[
\u = -\K \s + \Jh \x_r^{(n)} + \Y(\x, \ldots, \x^{(n-1)}) \ah - \frac{1}{2} \Jdh \s
\]
where $\x_r^{(n)} = \ddt \x_r^{(n-1)}$ provides the closed-loop dynamics
\begin{align}
\J \dot{\s} + \K \s 
= \W \at + \Q \ah
\label{eq:sdot}
\end{align}

Inspired by the elegant modification of the Slotine and Li adaptive robot controller introduced by Lee et al.~\cite{LeeKwonPark18,Lee19}, we consider the Lyapunov-like function
\[
V = \frac{1}{2} \s\T \J \s + d_f(\a || \ah)
\]
where $d_f(\a || \ah)$ denotes the Bregman divergence of a function $f$ assumed convex on $\Acal$ and given by
\[
d_f(\a || \ah) = f(\a) - f(\ah) - < \nabla f(\ah), \a- \ah>
\]
Note that the LogDet divergence from \eqref{eq:breg} follows this form for $f(\X) = - \logdet{\X}$.
Here, we consider the case when $\Acal$ is open and $f$ is chosen as a convex barrier function on $\Acal$, such that the Hessian metric $\H = \nabla^2 f(\x)$ endows $\Acal$ with a barrier Hessian manifold structure \cite{Lee19,nesterov2002riemannian}.
Note that if $f$ is a second-order function $\frac{1}{2}\a^T \P \a$, the Bregman divergence
is simply $\frac{1}{2}\at\T \P \at$, with $\H^{-1}$ equal to the constant matrix $\bf{P}^{-1}$ similar to the standard adaptive algorithm \cite{slotine1987adaptive}.

A quick calculation shows that the derivative of the Bregman divergence is simply $\dot{\ah}\T \H \ \at \ $, 
so that the adaptation law
\begin{equation}
\dot{\ah} = - \H^{-1} (\W + \Q)\T \ \s
\label{eq:ahat_dot}
\end{equation}
yields
\begin{align*}
\dot{V} &= - \s\T \K \s + \s\T (\W + \Q) \at + \dot{\ah}\T \H \ \at \ = \ - \s\T \K \s \ \le \ 0
\end{align*}

\noindent Considering a virtual system with $\W$ and $\Q$ as externally provided functions of time, the dynamics \eqref{eq:sdot} and \eqref{eq:ahat_dot} are equivalent to natural primal-dual over the function
\[
\script{L}  = \frac{1}{2} \s\T \K \s - \s\T \W \at  - \s\T \Q \ \ah
\]
in the decoupled metric $\M_\s = \J$ and $\M_{\ah} = \H(\ah)$. Overall, this construction enables the results in natural adaptive robot control~\cite{LeeKwonPark18,Lee19} to be extended to the broader class \eqref{eq:sys_for_adaptive}.
 

\begin{remark}
Note that a similar construction could be applied to provide natural adaptation within recent applications of nonlinear adaptive control \cite[Thm. 2]{LopezSlotine19} based on control contraction metrics \cite{manchester2017CCCM,singh2017robust}. 
\end{remark}

\subsection{Natural Primal Dual}

Continuous-time convex primal-dual optimization is analyzed from a nonlinear contraction perspective in~\cite{Kostya}, building on a earlier result of~\cite{modular}. As we now show, Theorem~\ref{th_main} yields a natural extension to geodesic primal-dual optimization, where convexity in terms of primal and dual variables is replaced by g-convexity, thus broadening the above results to state-dependent metrics.




\begin{theorem}\label{th_PD}
Consider a scalar function $\script{L}({\bf x},{\bf \llambda}, t)$,
with $\script{L}$ g-strongly convex over $\x$ and g-strongly concave over $\llambda$ in metrics 
$\M_\x(\x)$ and $\M_{\llambda}(\llambda)$ respectively. Then, the geodesic primal-dual dynamics~\eqref{eq:PD} is globally contracting, in metric 
\begin{equation}\label{GPD}
   {\bf M}({\bf x},{\bf \llambda})  = \begin{bmatrix}
     {\bf M_x}({\bf x}) & {\bf 0}\\
    {\bf 0} & {\bf M_\llambda}({\bf \llambda})
    \end{bmatrix}
\end{equation}
\end{theorem}

\begin{proof} \ Letting $\z = [{\x}\T, {\llambda}{}\T]\T$ and $\dot{\z}  = \f(\z,t)$ denote the overall system dynamics, the system's Jacobian can be written
\renewcommand{\L}{\script{L}}
\begin{equation} \label{eq:PD_Jacobian}
{\bf A}(\x,\llambda, t) =  \jac{\f}{\z} = 
\begin{bmatrix}
- \jac{}{\x} ( \M_\x^{-1} \grad{\L}{\x}) & - \ {\M_\x}^{-1}\ \hess{\L}{\x}{\llambda}\\[1ex]
   {\M_\llambda}^{-1} \ \hess{\L}{\llambda}{\x} & \jac{}{\llambda} (\M_\llambda^{-1} \grad{\L}{\llambda})
    \end{bmatrix} 
\end{equation}
\noindent so that, using Theorem~\ref{th_main}, 
\begin{equation}\nonumber
     {\bf M} {\bf A}  + \ {\bf A}\T {\bf M} \ + \ \dot{\bf M } \ = \ -2 \ \begin{bmatrix}
     {\bf H_x} & {\bf 0}\\
    {\bf 0} & - {\bf H_\llambda} 
    \end{bmatrix} \ < \ {\bf 0}
    \end{equation} \end{proof}

\begin{proposition}
\label{eq:semi_PD}
Consider the primal dual dynamics \eqref{eq:PD} for a scalar cost function $\script{L}({\bf x},{\bf \llambda})$,
with $\script{L}$ g-strongly convex over $\x$ and g-concave (not necessarily strongly so) over $\llambda$ in metrics 
$\M_\x(\x)$ and $\M_{\llambda}(\llambda)$ respectively. Suppose also that one solution of \eqref{eq:PD} is known to be bounded. Then, for any initial condition, the geodesic primal-dual dynamics~\eqref{eq:PD} converge to an
equilibrium $\x^*$, $\llambda^*$. Moreover, $\x^*$ is independent of initial conditions. 
\end{proposition}

\begin{proof}
The proof is given as a corollary to Theorem \ref{thm:semi_nash} in the next section.
\end{proof}

\begin{remark}
The above proposition highlights that contraction of the PD dynamics (e.g., as developed in \cite{Kostya}) is not necessary to guarantee convergence to a unique primal solution. Note however, that the above results only guarantee asymptotic convergence toward the unique primal equilibrium, as opposed to exponential convergence \cite{Kostya} when contraction can be shown for the PD dynamics as a whole. 
\end{remark}

This proposition is reminiscent of results in adaptive control wherein the error dynamics of a certainty-equivalent controller may be asymptotically stable despite the fact that an associated adaptation law may not converge to the actual unknown parameters~\cite{slotine1987adaptive,Lee18}, with adaptation occurring on a "need-to-know" basis in that sense. Conceptually, this principle can apply to more general contexts involving concurrent
control and learning, when effective control is the main goal (e.g., in reinforcement learning).

\begin{remark} Note that this analogy between adaptive control and primal-dual optimization also enables recent results in distributed adaptive control \cite{cloud} to be applied in a distributed primal-dual setting. These results also include straightforward strategies for stably handling communication delays (e.g., using a wave variable formulation \cite{Niemeyer91b,wang2006contraction}) which would find additional motivation in  distributed primal-dual optimization applications.
\end{remark}

\section{Applying Contraction Tools to G-Convex Optimization}
\label{sec:application}

Theorem \ref{th_main} immediately implies that existing combination properties \cite{Slotine98,modular} from contraction
analysis can be directly applied in the context of g-convex optimization.
While these properties derive from simple matrix algebra and in principle could be proven directly from the definition of geodesic convexity, as we will see most
rely for their practical relevance on the flexibility 
afforded by the contraction analysis point of view.

\subsection{Sum of g-convex}

If two functions $f_1(\x,t)$ and $f_2(\x,t)$ are g-convex in the same metric for each $t$, then their sum $f_1(\x,t) + f_2(\x,t)$ is g-convex in the same metric. 

\begin{example}
\label{example:Parallel_Geo_Desc}
Consider a function $f_1(\x_1,\y_1,t)$ g-convex for each $t$ in a block diagonal metric ${\rm BlkDiag}(\M_{x_1}(\x_1),\M_y(\y_1))$ and a function $f_2(\x_2,\y_2,t)$ g-convex for each $t$ in a block diagonal metric ${\rm BlkDiag}(\M_{x_2}(\x_2), \M_y(\y_2))$.
Then, the function:
\[
f(\x_1, \x_2, \y,t) =  f_1(\x_1,\y,t)+f_2(\x_2,\y,t)
\]
is g-convex in metric ${\rm BlkDiag}(\M_{x_1},\M_{x_2},\M_y)$ for each $t$.

\end{example}

\subsection{Skew-Symmetric Feedback Coupling}

Assume that a scalar function $f_1(\x_1,\x_2)$ is $\alpha_1$-strongly g-convex in $\x_1$ in a metric $\M_1(\x_1)$ for each fixed $\x_2$, and similarly that a scalar function $f_2(\x_1,\x_2)$ is $\alpha_2$-strongly g-convex in a metric $\M_2(\x_2)$ for each fixed $\x_1$. If $f_1$ and $f_2$ satisfy the scaled skew-symmetry property
\begin{equation}
\hess{f_1}{\x_1}{\x_2} \ = \ - \ k \ \hess{f_2}{\x_1}{\x_2} 
\label{eq:skew_property}
\end{equation}
where $k$ is some strictly positive constant, then the natural gradient dynamics 
\begin{align}
\label{eq:Nash_dyn}
\dot{\x}_1 &= - \M_1(\x_1)^{-1} \ \grad{f_1(\x_1,\x_2)}{\x_1}  \\
\dot{\x}_2 &= - \M_2(\x_2)^{-1} \ \grad{f_2(\x_1,\x_2)}{\x_2} \nonumber
\end{align}
is contracting with rate $ \ \min(\alpha_1, \alpha_2) $ in metric $\M(\x_1,\x_2) = {\rm BlkDiag}(\M_1(\x_1), k \M_2(\x_2))$. Since the {\em overall} system is both contracting and autonomous, it
tends to a unique equilibrium~\cite{Slotine98} $(\x_1^*, \x_2^*)$ which satisfies the Nash-like conditions 
\begin{align*}
\x_1^* &= {\rm argmin}_{\x_1} f_1(\x_1,\x_2^*) \\
\x_2^* &= {\rm argmin}_{\x_2} f_2(\x_1^*,\x_2) 
\end{align*}

Note that the result can be broadened to cases where the scaled skew-symmetry property is not exactly satisfied, by using the small-gain extension in~\cite{modular}. Taking again the machine learning context as a potential example, such two-player game dynamics can occur in certain types of adversarial training. 

The result extends to a game with an arbitrary number of players. Consider $n$ functions $\{f_i(\x_1,\ldots,\x_n)\}_{i=1}^n$ such that each $f_i$ is $\alpha_i$-strongly g-convex over $\x_i$ in a metric $\M_i(\x_i)$. If the functions satisfy the skew-symmetry conditions
\[
\hess{f_i}{\x_i}{\x_j} = - k_j \hess{f_j}{\x_i}{\x_j}
\]
for each $j>i$, then the suitable generalizations of \eqref{eq:Nash_dyn} result in a coupled system that is contracting with rate ${\rm min}(\alpha_1, \ldots, \alpha_n)$ in the metric 
\[
\M= {\rm BlkDiag}(\M_1, k_2 \M_2, \ldots, k_n \M_n).
\]
The overall system converges to a unique Nash-like equilibrium satisfying
\[
\x_1^* = {\rm argmin}_{\x_1} f_1(\x_1,\x_2^*,\ldots,\x_n^*)
\]
and a similar relation for each other player.

Likewise, the result can be extended to the case when the natural gradient dynamics for each individual player may only be semi-contracting.

\begin{theorem} 
\label{thm:semi_nash}
Consider the two player case \eqref{eq:Nash_dyn}, wherein (a) $f_1$ is $\alpha_1$-strongly g-convex with $\alpha_1>0$ in a uniformly positive definite metric $\M_1(\x_1)$ for each $\x_2$  (b) the Riemannian Hessian $\H_2(\x_1,\x_2)$ of $f_2(\x_1,\x_2)$ in $\x_2$ is only positive {\em semi}-definite for each $\x_1$ in a uniformly positive definite metric $\M_2(\x_2)$ and (c) the skew-symmetry property \eqref{eq:skew_property} holds. Assume that one trajectory of \eqref{eq:Nash_dyn} is known to be bounded. Then, every trajectory of \eqref{eq:Nash_dyn} converges to a Nash equilibrium $\x_1^*$, $\x_2^*$. Moreover, $\x_1^*$  does not depend on initial conditions (i.e., every Nash has the same strategy for player 1). 
\end{theorem}

\begin{proof}
It can be shown that virtual displacements evolve such that
\begin{align*}
\ddt &\left(\delta \x_1\T \M_1(\x_1)\delta \x_1 + k \delta \x_2\T \M_2(\x_2) \delta \x_2 \right)  \\
&\le \ - \ 2 \ \alpha_1 \delta\x_1 \T \M_1(\x_1) \delta \x_1 \ - \ 2 \ k \ \delta \x_2\T \H_2(\x_1,\x_2) \delta \x_2
\end{align*}
which implies, by Barbalat's lemma,
\[
\delta \x_1 \rightarrow 0 \quad {\rm and}\quad \H_2(\x_1,\x_2) \delta \x_2 \ \rightarrow \ 0
\]
Via the same argument as follows from \eqref{ddt_h}, it follows that $\grad{f_1}{\x_1}(\x_1(t), \x_2(t)) \rightarrow 0$ as $t\rightarrow \infty$. So, for each initial condition, $\x_1(t)$ must converge to some equilibrium $\x_1^*$ of the $\x_1$ dynamics. Furthermore, since any $ \delta \x_1 \rightarrow 0  $,
$\x_1^*$ must be {\em unique} and independent of initial conditions. Let us now turn to the behavior of the $\x_2$ dynamics. Given an arbitrary initial condition $(\x_{1,0}$, $\x_{2,0})$ for \eqref{eq:Nash_dyn}, let $L^+$ denote its $\omega$-limit set. Any point in $(\x_1,\x_2) \in L^+$ must satisfy $\x_1 = \x_1^*$. Since the dynamics are autonomous, $L^+$ is composed of trajectories of the system
\begin{align}
\dot{\x}_1 &= \mathbf{0} \\
\dot{\x}_2 &= - \M_2(\x_2)^{-1} \grad{f_2}{\x_2}(\x_1^*,\x_2)
\label{eq:x2_dyn}
\end{align}
Moreover, $L^+$ must be closed and bounded. Since \eqref{eq:x2_dyn} is a natural gradient system of a g-convex function, Thm.~\ref{thm:natGradNonStrict} ensures that any trajectory of \eqref{eq:x2_dyn} must converge to an equilibrium point that is a global minimizer for $f_2(\x_1^*, \x_2)$. Considering any initial condition of \eqref{eq:x2_dyn} that begins in $L^+$, we denote $(\x_1^*, \x_2^*) \in L^+$ as the resulting equilibrium point. However, since \eqref{eq:Nash_dyn} is semi-contracting, any geodesic ball around $(\x_1^*, \x_2^*)$ is forward invariant for \eqref{eq:Nash_dyn}, which implies that $L^+ = \{ (\x_1^*,\x_2^*) \}$. Thus, $\x_2(t)\rightarrow \x_2^*$ as $t\rightarrow \infty$.
Note again that, while $\x_2^*$ depends on initial conditions, $\x_1^*$ does not.
\end{proof}

\begin{corollary}
Consider any two equilibrium points $(\x_1^*, \x_{21}^*)$ and $(\x_1^*, \x_{22}^*)$ for a system that satisfies the condition of Theorem~\ref{thm:semi_nash}. Then, the geodesic between these points is comprised of extremal Nash equilibrium points, all of which have the same cost.

Further, if one equilibrium of \eqref{eq:x2_dyn} is locally asymptotically stable, then it is necessarily globally attractive, and thus all trajectories of \eqref{eq:x2_dyn} converge to this unique equilibrium $(\x_1^*, \x_2^*)$ regardless of initial conditions.
\end{corollary}
\begin{proof}
Proof of the first part follows immediately from applying 
Corollary 3.1 of \cite{RAPCSAK91} to the function $f_2(\x_1^*, \x_2)$. Proof of the second part follows immediately from the application of Corollary \ref{think_locally} herein.
\end{proof}

\begin{corollary}
Proposition \ref{eq:semi_PD} is true.
\end{corollary}
\begin{proof}
Consider Thm.~\ref{thm:semi_nash} with $f_1 = \script{L}(\x,\llambda)$ and $f_2 = - \script{L}(\x,\llambda)$.
\end{proof}

\subsection{Hierarchical Natural Gradient}
Consider a function $f_1(\x_1)$ $\alpha_1$-strongly g-convex in a metric $\M_1(\x_1)$, and a function $f_2(\x_1,\x_2)$ $\alpha_2$-strongly g-convex in a metric $\M_2(\x_2)$ for each given $\x_1$. Then, the hierarchical natural gradient dynamics
\begin{align*}
\dot{\x}_1 &= - \M_1(\x_1)^{-1} \ \grad{f_1(\x_1)}{\x_1}  \\
\dot{\x}_2 &= - \M_2(\x_2)^{-1} \ \grad{f_2(\x_1,\x_2)}{\x_2} 
\end{align*}
is contracting with rate $ \ \min(\alpha_1, \alpha_2) \ $ in metric $\M(\x_1,\x_2) = {\rm BlkDiag}(\M_1(\x_1), \M_2(\x_2))$, under
the mild assumption that the coupling Jacobian
is bounded~\cite{Slotine98}. Since the {\em overall} system is both contracting and autonomous, it
tends to a unique equilibrium~\cite{Slotine98} at rate $\min(\alpha_1, \alpha_2)$, and thus to the unique solution of
\begin{align*}
\grad{f_1(\x_1)}{\x_1} &= {\bf 0}\\
\grad{f_2(\x_1,\x_2)}{\x_2}  &= {\bf 0}
\end{align*}

By recursion, this structure can be chained an arbitrary number of times, or applied to any cascade or directed acyclic graph of natural gradient dynamics. Such hierarchical optimization may play a role, for instance, in backpropagation of natural gradients in machine learning, with all descents occurring concurrently rather than in sequence.

\begin{remark}
In large-scale optimization settings such as those appearing commonly in machine learning, natural gradient with a fully-dense metric can become intractable. In specific cases, such as natural gradient descent based on Fisher information~\cite{Amari98b},
computationally effective approximations have been derived~\cite{Martens15,Amari2018}.
In addition, the combination of simple (e.g., diagonal) metrics through hierarchical structures lends an opportunity to recover significant complexity at broad scale $-$ see, e.g., the hierarchical combination of scalar metrics to learn hierarchical representations of symbolic data in  \cite{Nickel17,Nickel18}. Such simpler metrics are also 
well motivated in the context of positive or monotone
systems \cite{Rantzer15,Manchester17}. In special cases of a dense Hessian metric $\M(\x) = \nabla^2 \psi(\x)$ from a potential $\psi(\x)$, note that continuous mirror descent (see also Proposition~\ref{change_variable} and Example~\ref{psi}) provides an alternate method to compute continuous natural gradient. These methods can avoid the need to invert the metric in cases where there is an explicit inverse exists for the change of variables $\z = \nabla \psi(\x)$, or when \eqref{eq:InvertGrandientViaDynamics} can be run at a fast time scale to invert the gradient map through dynamics.
\end{remark}

\section{Conclusions}
\label{sec:Conclusions}

Overall, this paper has demonstrated that nonlinear contraction analysis provides a  general perspective for analyzing and certifying the global convergence properties of gradient-based optimization algorithms. The common case of strong convexity corresponds to the special case of contracting gradient descent in the identity metric, while our analysis admits global convergence results in the significantly broader case of state-dependent metrics. This result has clear links to the case of geodesically-convex optimization wherein natural gradient descent converges to a unique equilibrium if it is contracting in any metric, broadening from the special case of g-convexity corresponding to contraction in the natural metric. Our analysis of semi-contraction of gradient systems, and the resulting smoothly connected sets of global optima may shed additional light on applications in learning with over parameterized networks \cite{cooper2018loss} where the set of optimizers is recognized to take the form of a low-dimensional manifold. Results on natural primal dual and the convergence to Nash equilibria showcase the broad reach of these fundamental results, where they may serve as the basis for the generation of larger scale distributed optimization algorithms in future work. 
A framework we call Contraction +  shows how contraction or semi-contraction properties can be combined with specific but coarse information on a system, such as the local stability of a particular equilibrium or the weak decreasing of a cost or a Lyapunov-like function, to conclude on global convergence or minimization.

A natural next step for the application of contraction in optimization is to design geodesic quorum sensing ~\cite{tabareau,russo} algorithms for synchronization \cite{bouvrie}, as well as other consensus mechanisms considering time-delays~\cite{Wei_delay,cloud}, which may serve as the basis for distributed and large-scale optimization techniques on Riemannian manifolds. Other future applications will consider stochastic gradient descent in the Riemannian setting~\cite{Bonnabel13} with quorum sensing extensions (as, e.g., in~\cite{Lecun,Boffi}). Such advances could have direct applications, e.g., in the context of machine learning, among others.


\noindent {\bf Acknowledgements} \ \  We thank Nicholas Boffi for stimulating discussions. This research was supported in part by grant 1809314 from the National Science Foundation.

\ifarxiv
\else
\newpage
\fi

\section*{Supporting Information}
\paragraph*{Proof of Theorem 1}

\setcounter{proposition}{0}
\renewcommand*{\theproposition}{A.\arabic{proposition}}
\newcommand*{\thepropositiondis}{A.\arabic{proposition}}

\setcounter{remark}{0}
\renewcommand*{\theremark}{A.\arabic{remark}}
\newcommand*{\theremarkdis}{A.\arabic{remark}}

\newcommand{\nablaM}{\overset{\scriptscriptstyle \M}{\nabla}{}}

We precede the proof of Theorem 1 with a more general result that links contraction with the notion of taking covariant derivatives \cite{do1992riemannian, lovelock1989tensors}.

If we view $\h(\x,t) :\R^n \times \R \rightarrow \R^n$ as providing the components of a vector field, then its covariant derivative (w.r.t.~the Riemannian connection $\nablaM{}$ \cite{do1992riemannian}) along the $i$-th coordinate vector field has components \cite{lovelock1989tensors}:
\begin{equation}
\Big[ \nablaM{}_{\partial_i} \h \Big]_j = \pd{}{i}{h_j} + \Gamma_{ki}^j  h_k
\label{eq:covariant}
\end{equation}
where $\Gamma_{ij}^k$ denotes the Christoffel symbol of the second kind
\begin{equation}
\label{eq:christoffel}
\Gamma_{ij}^m = \frac{1}{2} M^{mk} \left( \pd{M_{ik}}{\xsubj} + \pd{M_{jk}}{\xsubi} - \pd{M_{ij}}{\xsubk} \right)
\end{equation}
and the usual Einstein summation convention is applied  (implying, e.g., a sum over $k$ in the above equations).

By contrast, if we view a function $\g(\x,t) :\R^n \times \R \rightarrow \R^n$ as giving the components of a co-vector field (i.e., a covariant vector field) we have \cite{lovelock1989tensors}:
\[
\Big[ \nablaM{}^*_{\partial_i} \g \Big]_j = \pd{}{i}{g_j} - \Gamma_{ij}^k  g_k
\]
where the $^*$ on $\nablaM{}^*$ denotes that we are considering $\g$ as giving the components of a co-vector field. In either case, we can collect these derivatives along each coordinate vector field into Jacobian-like matrices denoted:
\[
\nablaM\, \h = \begin{bmatrix}\nablaM_{\partial_1}\h & \cdots & \nablaM_{\partial_n} \h \end{bmatrix} ~~~~~~~~~
\nablaM^* \, \g = \begin{bmatrix}\nablaM^*_{\partial_1}\g & \cdots & \nablaM^*_{\partial_n} \g \end{bmatrix}
\]
We are now prepared to state the main Proposition involved in proving Theorem 1.

\begin{proposition}
\label{prop:appendix}
Consider a dynamical system $\dot{\x} = \h(\x,t)$ which we re-write in the form:
\[
\M(\x) \dot{\x} = \g(\x,t) 
\]
where $\h= \M^{-1} \g$. Then, we have the equivalence:
\begin{align}
\M \left(\Jac{\h}{\x} \right) + \left( \Jac{\h}{\x} \right)\T \M + \dot{\M} &= \left[ \nablaM^* \g \right]  + \left[\nablaM^* \g \right]\T  
\label{eq:covar_result}
\end{align}
\end{proposition}

\begin{proof}

To assist the proof, we define
\[
\mathbf{Q} = \M \left(\Jac{\h}{\x} \right) + \left( \Jac{\h}{\x} \right)\T \M + \dot{\M}
\]

First consider the partials of $\h = \M^{-1} \g$,
\begin{align*}
\pd{h_k}{\xsub{j}} &= \pd{}{\xsub{j}}\left[  M^{k\ell}\, g_\ell \right] \nonumber \\
			  &= M^{k\ell} ( \pd{g_\ell}{j})  - M^{kr} (\pd{M_{rs}}{\xsub{j} }) M^{s\ell} g_\ell
\end{align*}

Using this result
\begin{align*}
Q_{ij} =& M_{ik} (\pd{h_k}{\xsub{j}}) + M_{jk} (\pd{h_k}{\xsub{i}}) + (\pd{M_{ij}}{\xsub{k}}) M^{k\ell} g_\ell \\
			   =& M_{ik} M^{k\ell} (\pd{g_\ell}{j})  - M_{ik} M^{kr} (\pd{M_{rs}}{\xsub{j}}) M^{s\ell} g_\ell \\
			    &+M_{jk} M^{k\ell} (\pd{g_\ell}{i})  - M_{jk} M^{kr} (\pd{M_{rs}}{\xsub{i}}) M^{s\ell} g_\ell \\
			    &+(\pd{M_{ij}}{\xsub{k}}) M^{k\ell} g_\ell 
\end{align*}
Noting that $M_{ik} M^{kj} = \delta_{ij}$, with $\delta_{ij}$ the Kronecker delta,
\begin{align*}
Q_{ij} &= \pd{g_i}{j} + \pd{g_j}{i} - M^{\ell k} \left[ \pd{M_{ik}}{\xsub{j}} + \pd{M_{jk}}{\xsub{i}} -\pd{M_{ij}}{\xsub{k}} \right] g_\ell \\
			   &= \pd{g_i}{j} + \pd{g_j}{i} - 2 \Gamma^\ell_{ij} g_\ell \\
			   &= \left[\nablaM^* \g\right]_{ij} + \left[\nablaM^* \g\right ]_{ji}			    \nonumber
\end{align*}
Where the last line follows since $\Gamma_{ij}^\ell = \Gamma_{ji}^\ell$ for the symbols of the Riemmanian connection $\nablaM$.
\end{proof}

\begin{remark}
Following the setup for the previous proposition, the covariant derivatives can also be shown to satisfy:
\begin{equation}
\M \left[ \nablaM \h \right]  = \nablaM^* \g
\label{eq:co_vs_contra}
\end{equation}
which immediately gives:
\begin{equation}
 \M \left(\Jac{\h}{\x} \right) + \left( \Jac{\h}{\x} \right)\T \M + \dot{\M} = \M \left[ \nablaM \h \right] +  \left[ \nablaM \h \right]^T \M \label{eq:covariant_contraction}
\end{equation}
\end{remark}

\newcommand{\Ttheta}{\boldsymbol{\Theta}}
\newcommand{\Tthetadot}{\dot{\Ttheta}}

\begin{remark}
We consider the metric factored as $\M = \Ttheta\T \Ttheta$ with the generalized Jacobian defined according to \cite{Slotine98}
\[
\mathbf{F} = \Ttheta \frac{\partial \h}{\partial \x} \Ttheta^{-1} + \Tthetadot \Ttheta^{-1}
\]
where we associate $\delta \z = \Theta \delta \x$ as a differential change of coordinates. The generalized Jacobian satisfies that, along the flow of the system, local perturbations evolve according to $\frac{\rm d}{{\rm d} t}\delta \z = \mathbf{F} \delta \z$ \cite{Slotine98}.

Considering a similarity transform on $\mathbf{F}$
\[
\Ttheta^{-1} \, \mathbf{F} \, \Ttheta = \frac{\partial \h}{\partial \x} + \Ttheta^{-1} \Tthetadot
\]
we see that the right-hand side takes a similar structural form to the covariant derivative \eqref{eq:covariant}, as noted originally in \cite{Slotine98}. Indeed, defining some new connection coefficients:
\[
\tilde{\Gamma}^{i}_{jk} = \left[ \Ttheta^{-1} \frac{\partial \Ttheta}{\partial x_j}  \right]_{ik}
\]
determines an affine connection with covariant derivatives uniquely specified according to 
\[
\tilde{\nabla}_{\partial_i} \mathbf{e}_j = \tilde{\Gamma}^{k}_{ij} \mathbf{e}_k
\]
where $\mathbf{e}_j$ gives the $j$-th unit vector. (Note this equation uses the symbols to sum over vectors, as opposed to components of vectors as previously.) While this new connection can be shown to be metric compatible (as defined in \cite{do1992riemannian}), it is not in general equal to the one provided by the Riemannian connection $\nablaM$ with the corresponding symbols~$\Gamma^{i}_{jk}$ via \eqref{eq:christoffel}. Any differences do not affect the final contraction analysis in the sense that relationship \eqref{eq:covariant_contraction} still holds when the Riemannian connection $\nablaM$ is replaced with any metric-compatible connection. However, \eqref{eq:co_vs_contra} and therefore \eqref{eq:covar_result} does not hold in general. 
\end{remark}

We are now ready to prove Theorem 1 from the main text.
\begin{proof}[Proof of Theorem 1]
Recall that $\alpha$-strong geodesic convexity of $f(\x,t)$ in the metric $ \M(\x)$ (for each $t$) is equivalent to the Riemannian Hessian of $f$, denoted $\H(\x,t)$, satisfying:
\[
\H(\x,t) \succeq \alpha \M(\x) \quad \forall \x
\]

In coordinates, entries of $\H$ are given by
\[
H_{ij} = \pdd{f}{\xsubi}{\xsubj} - \Gamma_{ij}^k \left( \pd{f}{\xsub{k}} \right) 
\]
which we see is directly related to taking the covariant derivative of the co-vector field with components $\partial_k f$. More specifically, defining $\g(\x,t) = -\nabla f(\x,t)$ we have:
\[
\H = -\nablaM^* \g
\]

Via Proposition~\ref{prop:appendix}, we thus have $\mathbf{Q} = \M \left(\Jac{\h}{\x} \right) + \left( \Jac{\h}{\x} \right)\T \M + \dot{\M} = -2 \mathbf{H}$ such that contraction of the natural gradient dynamics \eqref{eq:NatGrad} with rate $\alpha$ under the inequality
\[
\Q \preceq - 2 \alpha \M
\]
is equivalent to requiring $\H \succeq \alpha \M$.


\end{proof}

\nolinenumbers

\ifarxiv
	\bibliography{Geodesic_plos_final.bbl}
\else
	\bibliographystyle{plos2015}
	\bibliography{Geodesic}

\begin{thebibliography}{10}

\bibitem{bassily2018exponential}
Bassily R, Belkin M, Ma S.
\newblock On exponential convergence of sgd in non-convex over-parametrized
  learning.
\newblock arXiv preprint arXiv:181102564. 2018;.

\bibitem{cooper2018loss}
Cooper Y.
\newblock The loss landscape of overparameterized neural networks.
\newblock arXiv preprint arXiv:180410200. 2018;.

\bibitem{liu2020theory}
Liu C, Zhu L, Belkin M.
\newblock Toward a theory of optimization for over-parameterized systems of
  non-linear equations: the lessons of deep learning.
\newblock arXiv preprint arXiv:200300307. 2020;.

\bibitem{brea2019weight}
Brea J, Simsek B, Illing B, Gerstner W.
\newblock Weight-space symmetry in deep networks gives rise to permutation
  saddles, connected by equal-loss valleys across the loss landscape.
\newblock arXiv preprint arXiv:190702911. 2019;.

\bibitem{sagun2017empirical}
Sagun L, Evci U, Guney VU, Dauphin Y, Bottou L.
\newblock Empirical analysis of the hessian of over-parametrized neural
  networks.
\newblock arXiv preprint arXiv:170604454. 2017;.

\bibitem{allen2018convergence}
Allen-Zhu Z, Li Y, Song Z.
\newblock A convergence theory for deep learning via over-parameterization.
\newblock arXiv preprint arXiv:181103962. 2018;.

\bibitem{du2018gradient}
Du SS, Zhai X, Poczos B, Singh A.
\newblock Gradient descent provably optimizes over-parameterized neural
  networks.
\newblock arXiv preprint arXiv:181002054. 2018;.

\bibitem{hanson1981sufficiency}
Hanson MA.
\newblock On sufficiency of the Kuhn-Tucker conditions.
\newblock Journal of Mathematical Analysis and Applications.
  1981;80(2):545--550.

\bibitem{zualinescu2014critical}
Zalinescu C.
\newblock A critical view on invexity.
\newblock Journal of Optimization Theory and Applications.
  2014;162(3):695--704.

\bibitem{polyak1963gradient}
Polyak BT.
\newblock Gradient methods for minimizing functionals.
\newblock Zhurnal Vychislitel'noi Matematiki i Matematicheskoi Fiziki.
  1963;3(4):643--653.

\bibitem{karimi2016linear}
Karimi H, Nutini J, Schmidt M.
\newblock Linear convergence of gradient and proximal-gradient methods under
  the polyak-{\l}ojasiewicz condition.
\newblock In: Joint European Conference on Machine Learning and Knowledge
  Discovery in Databases. Springer; 2016. p. 795--811.

\bibitem{RAPCSAK91}
Rapcsak T.
\newblock Geodesic Convexity in Nonlinear Optimization.
\newblock Journal of Optimization Theory and Applications. 1991;69(1):169--183.

\bibitem{Sra_First_Order_Methods}
Zhang H, Sra S.
\newblock First-order methods for geodesically convex optimization.
\newblock In: Conference on Learning Theory; 2016. p. 1617--1638.

\bibitem{Absil08}
Absil PA, Mahony R, Sepulchre R.
\newblock Optimization Algorithms on Matrix Manifolds.
\newblock Princeton University Press; 2008.

\bibitem{Slotine98}
Lohmiller W, Slotine JJE.
\newblock On Contraction Analysis for Non-linear Systems.
\newblock Automatica. 1998;34(6):683--696.
\newblock doi:{10.1016/S0005-1098(98)00019-3}.

\bibitem{tabareau}
Tabareau N, Slotine JJ, Pham QC.
\newblock How synchronization protects from noise.
\newblock PLoS computational biology. 2010;6(1):e1000637.

\bibitem{Wei_delay}
Wang W, Slotine JJE.
\newblock Contraction analysis of time-delayed communications and group
  cooperation.
\newblock IEEE Transactions on Automatic Control. 2006;51(4):712--717.
\newblock doi:{10.1109/TAC.2006.872761}.

\bibitem{cloud}
{Wensing} PM, {Slotine} JJE.
\newblock {Cooperative Adaptive Control for Cloud-Based Robotics}.
\newblock IEEE International Conference on Robotics and Automation. 2018;.

\bibitem{modular}
Slotine JJE.
\newblock Modular stability tools for distributed computation and control.
\newblock International Journal of Adaptive Control and Signal Processing.
  2003;17(6):397--416.

\bibitem{su2014differential}
Su W, Boyd S, Candes E.
\newblock A differential equation for modeling Nesterov's accelerated gradient
  method: Theory and insights.
\newblock In: Advances in Neural Information Processing Systems; 2014. p.
  2510--2518.

\bibitem{Zhang18}
{Zhang} J, {Mokhtari} A, {Sra} S, {Jadbabaie} A.
\newblock {Direct Runge-Kutta Discretization Achieves Acceleration}.
\newblock ArXiv e-prints. 2018;.

\bibitem{wibisono2016variational}
Wibisono A, Wilson AC, Jordan MI.
\newblock A variational perspective on accelerated methods in optimization.
\newblock Proceedings of the National Academy of Sciences.
  2016;113(47):E7351--E7358.

\bibitem{krichene2015accelerated}
Krichene W, Bayen A, Bartlett PL.
\newblock Accelerated mirror descent in continuous and discrete time.
\newblock In: Advances in neural information processing systems; 2015. p.
  2845--2853.

\bibitem{nesterov1998introductory}
Nesterov Y.
\newblock Introductory lectures on convex programming -- A Basic course.
\newblock Springer; 1998.

\bibitem{Kostya}
Nguyen HD, Vu TL, Turitsyn K, Slotine JJ.
\newblock Contraction and Robustness of Continuous Time Primal-Dual Dynamics.
\newblock IEEE Control Systems Letters. 2018;2(4).

\bibitem{frana2020dissipative}
França G, Jordan MI, Vidal R.
\newblock On Dissipative Symplectic Integration with Applications to
  Gradient-Based Optimization.
\newblock arXiv preprint arXiv:200406840. 2020;.

\bibitem{slotine1991applied}
Slotine JJE, Li W.
\newblock Applied nonlinear control.
\newblock Prentice hall Englewood Cliffs, NJ; 1991.

\bibitem{Wiggins2003}
Wiggins S.
\newblock Gradient Vector Fields.
\newblock In: Introduction to Applied Nonlinear Dynamical Systems and Chaos.
  New York, NY: Springer New York; 2003. p. 231--233.

\bibitem{aylward2008stability}
Aylward EM, Parrilo PA, Slotine JJE.
\newblock Stability and robustness analysis of nonlinear systems via
  contraction metrics and SOS programming.
\newblock Automatica. 2008;44(8):2163--2170.

\bibitem{Boyd_GeomProgram}
Boyd S, Kim SJ, Vandenberghe L, Hassibi A.
\newblock A tutorial on geometric programming.
\newblock Optimization and Engineering. 2007;8(1):67.
\newblock doi:{10.1007/s11081-007-9001-7}.

\bibitem{SRA_Conic_Optimization}
Sra S, Hosseini R.
\newblock Conic Geometric Optimization on the Manifold of Positive Definite
  Matrices.
\newblock SIAM Journal on Optimization. 2015;25(1):713--739.
\newblock doi:{10.1137/140978168}.

\bibitem{udriste}
Udriste C.
\newblock Convex functions and optimization methods on Riemannian manifolds.
  vol. 297.
\newblock Springer Science \& Business Media; 1994.

\bibitem{Amari98b}
Amari SI.
\newblock Natural Gradient Works Efficiently in Learning.
\newblock Neural Comput. 1998;10(2):251--276.
\newblock doi:{10.1162/089976698300017746}.

\bibitem{gunasekar2018characterizing}
Gunasekar S, Lee J, Soudry D, Srebro N.
\newblock Characterizing implicit bias in terms of optimization geometry.
\newblock arXiv preprint arXiv:180208246. 2018;.

\bibitem{lovelock1989tensors}
Lovelock D, Rund H.
\newblock Tensors, differential forms, and variational principles.
\newblock Courier Corporation; 1989.

\bibitem{lohmiller2013exact}
Lohmiller W, Slotine JJ.
\newblock Exact decomposition and contraction analysis of nonlinear hamiltonian
  systems.
\newblock In: AIAA Guidance, Navigation, and Control (GNC) Conference; 2013. p.
  4931.

\bibitem{Loh09}
Lohmiller W, Slotine JJE.
\newblock Exact Modal Decomposition of Nonlinear Hamiltonian Systems.
\newblock In: AIAA Guidance, Navigation, and Control Conference; 2009. p.
  5792:1--18.

\bibitem{bullo}
Simpson-Porco JW, Bullo F.
\newblock Contraction theory on Riemannian manifolds.
\newblock Systems \& Control Letters. 2014;65:74--80.

\bibitem{dauphin2014identifying}
Dauphin YN, Pascanu R, Gulcehre C, Cho K, Ganguli S, Bengio Y.
\newblock Identifying and attacking the saddle point problem in
  high-dimensional non-convex optimization.
\newblock In: Advances in neural information processing systems; 2014. p.
  2933--2941.

\bibitem{jin2017escape}
Jin C, Ge R, Netrapalli P, Kakade SM, Jordan MI.
\newblock How to escape saddle points efficiently.
\newblock In: Proceedings of the 34th International Conference on Machine
  Learning-Volume 70. JMLR. org; 2017. p. 1724--1732.

\bibitem{lee2016gradient}
Lee JD, Simchowitz M, Jordan MI, Recht B.
\newblock Gradient Descent Converges to Minimizers.
\newblock arXiv preprint arXiv:160204915. 2016;.

\bibitem{lee2017first}
Lee JD, Panageas I, Piliouras G, Simchowitz M, Jordan MI, Recht B.
\newblock First-order methods almost always avoid saddle points.
\newblock arXiv preprint arXiv:171007406. 2017;.

\bibitem{kreusser2019deterministic}
Kreusser LM, Osher SJ, Wang B.
\newblock A Deterministic Approach to Avoid Saddle Points.
\newblock arXiv preprint arXiv:190106827. 2019;.

\bibitem{lohmiller2009exact}
Lohmiller W, Slotine JJ.
\newblock Exact Modal Decomposition of Nonlinear Hamiltonian Systems.
\newblock In: AIAA Guidance, Navigation, and Control Conference; 2009. p. 5792.

\bibitem{BoydVanberghe04}
Boyd S, Vandenberghe L.
\newblock Convex Optimization.
\newblock Cambridge, U.K.: Cambridge Univ. Press; 2004.

\bibitem{wei}
Wang W, Slotine JJE.
\newblock On partial contraction analysis for coupled nonlinear oscillators.
\newblock Biological cybernetics. 2005;92(1):38--53.

\bibitem{jouffroy}
Jouffroy J, Slotine JJE.
\newblock Methodological remarks on contraction theory.
\newblock In: IEEE Conference on Decision and Control. vol.~3; 2004. p.
  2537--2543 Vol.3.

\bibitem{manchester2017CCCM}
Manchester IR, Slotine JJE.
\newblock Control Contraction Metrics: Convex and Intrinsic Criteria for
  Nonlinear Feedback Design.
\newblock IEEE Transactions on Automatic Control. 2017;62(6):3046--3053.
\newblock doi:{10.1109/TAC.2017.2668380}.

\bibitem{bulloSemi2020}
Cisneros-Velarde P, Jafarpour S, Bullo F.
\newblock Distributed and time-varying primal-dual dynamics via contraction
  analysis.
\newblock arXiv preprint arXiv:200312665. 2020;.

\bibitem{singh2017robust}
Singh S, Majumdar A, Slotine JJ, Pavone M.
\newblock Robust online motion planning via contraction theory and convex
  optimization.
\newblock In: 2017 IEEE International Conference on Robotics and Automation
  (ICRA). IEEE; 2017. p. 5883--5890.

\bibitem{Manchester14_transverse}
Manchester IR, Slotine JJE.
\newblock Transverse contraction criteria for existence, stability, and
  robustness of a limit cycle.
\newblock Systems \& Control Letters. 2014;63:32 -- 38.
\newblock doi:{10.1016/j.sysconle.2013.10.005}.

\bibitem{deep_residual}
He K, Zhang X, Ren S, Sun J.
\newblock Deep residual learning for image recognition.
\newblock In: Proceedings of the IEEE conference on computer vision and pattern
  recognition; 2016. p. 770--778.

\bibitem{neural_ode}
Chen TQ, Rubanova Y, Bettencourt J, Duvenaud DK.
\newblock Neural ordinary differential equations.
\newblock In: Advances in neural information processing systems; 2018. p.
  6571--6583.

\bibitem{dupont}
Dupont E, Doucet A, Teh YW.
\newblock Augmented neural odes.
\newblock In: Advances in Neural Information Processing Systems; 2019. p.
  3134--3144.

\bibitem{arrow1958studies}
{K J  Arrow, L  Hurwicz, and H  Uzawa}.
\newblock Studies in Linear and Non-linear Programming.
\newblock Stanford University Press, Stanford, CA; 1958.

\bibitem{paganini}
Feijer D, Paganini F.
\newblock Stability of primal--dual gradient dynamics and applications to
  network optimization.
\newblock Automatica. 2010;46(12):1974 -- 1981.
\newblock doi:{10.1016/j.automatica.2010.08.011}.

\bibitem{madry2017towards}
Madry A, Makelov A, Schmidt L, Tsipras D, Vladu A.
\newblock Towards deep learning models resistant to adversarial attacks.
\newblock arXiv preprint arXiv:170606083. 2017;.

\bibitem{tishby2015deep}
Tishby N, Zaslavsky N.
\newblock Deep learning and the information bottleneck principle.
\newblock In: 2015 IEEE Information Theory Workshop (ITW). IEEE; 2015. p. 1--5.

\bibitem{Wang17}
Cho WS, Wang M.
\newblock Deep Primal-Dual Reinforcement Learning: Accelerating Actor-Critic
  using Bellman Duality.
\newblock CoRR. 2017;abs/1712.02467.

\bibitem{kosaraju2018primal}
Kosaraju KC, Mohan S, Pasumarthy R.
\newblock On the primal-dual dynamics of Support Vector Machines.
\newblock International Symposium on Mathematical Theory of Networks and
  Systems. 2018; p. 468--474.

\bibitem{ortega2003power}
Ortega R, Jeltsema D, Scherpen JM.
\newblock Power shaping: A new paradigm for stabilization of nonlinear RLC
  circuits.
\newblock IEEE Transactions on Automatic Control. 2003;48(10):1762--1767.

\bibitem{cavanagh2018transient}
Cavanagh K, Belk JA, Turitsyn K.
\newblock Transient stability guarantees for ad hoc DC microgrids.
\newblock IEEE Control Systems Letters. 2018;2(1):139--144.

\bibitem{slotine1987adaptive}
Slotine JJE, Li W.
\newblock On the adaptive control of robot manipulators.
\newblock The international journal of robotics research. 1987;6(3):49--59.

\bibitem{LeeKwonPark18}
{Lee} T, {Kwon} J, {Park} FC.
\newblock A Natural Adaptive Control Law for Robot Manipulators.
\newblock In: 2018 IEEE/RSJ International Conference on Intelligent Robots and
  Systems (IROS); 2018. p. 1--9.

\bibitem{Lee19}
Lee T.
\newblock Geometric Methods for Dynamic Model-Based Identification and Control
  of Multibody Systems.
\newblock Seoul National University; 2019.

\bibitem{nesterov2002riemannian}
Nesterov YE, Todd MJ, et~al.
\newblock On the Riemannian geometry defined by self-concordant barriers and
  interior-point methods.
\newblock Foundations of Computational Mathematics. 2002;2(4):333--361.

\bibitem{LopezSlotine19}
{Lopez} BT, {Slotine} JJE.
\newblock {Contraction Metrics in Adaptive Nonlinear Control}.
\newblock arXiv e-prints. 2019; p. arXiv:1912.13138.

\bibitem{Lee18}
{Lee} T, {Kwon} J, {Park} FC.
\newblock A Natural Adaptive Control Law for Robot Manipulators.
\newblock In: 2018 IEEE/RSJ International Conference on Intelligent Robots and
  Systems (IROS); 2018. p. 1--9.

\bibitem{Niemeyer91b}
Niemeyer G, Slotine JJE.
\newblock Stable adaptive teleoperation.
\newblock IEEE Journal of Oceanic Engineering. 1991;16(1):152--162.

\bibitem{wang2006contraction}
Wang W, Slotine JJ.
\newblock Contraction analysis of time-delayed communications and group
  cooperation.
\newblock IEEE Transactions on Automatic Control. 2006;51(4):712--717.

\bibitem{Martens15}
Martens J, Grosse R.
\newblock Optimizing Neural Networks with Kronecker-factored Approximate
  Curvature.
\newblock In: Proceedings of the 32nd International Conference on International
  Conference on Machine Learning; 2015. p. 2408--2417.

\bibitem{Amari2018}
Amari Si, Karakida R, Oizumi M.
\newblock Information geometry connecting Wasserstein distance and
  Kullback--Leibler divergence via the entropy-relaxed transportation problem.
\newblock Information Geometry. 2018;1(1):13--37.
\newblock doi:{10.1007/s41884-018-0002-8}.

\bibitem{Nickel17}
Nickel M, Kiela D.
\newblock Poincar{\'{e}} Embeddings for Learning Hierarchical Representations.
\newblock CoRR. 2017;abs/1705.08039.

\bibitem{Nickel18}
{Nickel} M, {Kiela} D.
\newblock {Learning Continuous Hierarchies in the Lorentz Model of Hyperbolic
  Geometry}.
\newblock ArXiv e-prints. 2018;.

\bibitem{Rantzer15}
Rantzer A.
\newblock Scalable control of positive systems.
\newblock European Journal of Control. 2015;24:72 -- 80.
\newblock doi:{https://doi.org/10.1016/j.ejcon.2015.04.004}.

\bibitem{Manchester17}
Manchester IR, Slotine JJE.
\newblock On Existence of Separable Contraction Metrics for Monotone Nonlinear
  Systems.
\newblock IFAC-PapersOnLine. 2017;50(1):8226 -- 8231.
\newblock doi:{https://doi.org/10.1016/j.ifacol.2017.08.1389}.

\bibitem{russo}
Russo G, Slotine JJE.
\newblock Global convergence of quorum-sensing networks.
\newblock Physical Review E. 2010;82(4):041919.

\bibitem{bouvrie}
Bouvrie J, Slotine JJ.
\newblock Synchronization Can Control Regularization in Neural Systems via
  Correlated Noise Processes.
\newblock In: Proceedings of the 25th International Conference on Neural
  Information Processing Systems. USA; 2012. p. 854--862.

\bibitem{Bonnabel13}
Bonnabel S.
\newblock Stochastic Gradient Descent on Riemannian Manifolds.
\newblock IEEE Transactions on Automatic Control. 2013;58(9):2217--2229.
\newblock doi:{10.1109/TAC.2013.2254619}.

\bibitem{Lecun}
Zhang S, Choromanska A, LeCun Y.
\newblock Deep Learning with Elastic Averaging SGD.
\newblock In: Proceedings of the 28th International Conference on Neural
  Information Processing Systems. Cambridge, MA, USA: MIT Press; 2015. p.
  685--693.

\bibitem{Boffi}
Boffi NM, Slotine JJE.
\newblock A Continuous-Time Analysis of Distributed Stochastic Gradient.
\newblock Neural Computation. 2020;32(1):36--96.
\newblock doi:{10.1162/neco\_a\_01248}.

\bibitem{do1992riemannian}
Do~Carmo MP, Flaherty~Francis J.
\newblock Riemannian geometry. vol.~6.
\newblock Springer; 1992.

\end{thebibliography}
\fi

\end{document}